\documentclass{article}
\usepackage[utf8]{inputenc}
\usepackage[margin=1 in]{geometry}
\usepackage{amsmath}
\usepackage{amsfonts}
\usepackage{mathtools}
\usepackage{amssymb}
\usepackage{amsthm}
\usepackage{mathrsfs}
\usepackage{tikz-cd}
\usepackage{enumitem}
\usepackage{youngtab}
\usepackage{lmodern}
\usepackage{hyperref}
\usepackage{todonotes}
\usepackage{circledsteps}

\usepackage{biblatex}
\newtheorem{theorem}{Theorem}
\numberwithin{theorem}{section}
\newtheorem{lemma}[theorem]{Lemma}
\newtheorem{proposition}[theorem]{Proposition}
\newtheorem{corollary}[theorem]{Corollary}
\newtheorem{claim}{Claim}
\newtheorem{conjecture}[theorem]{Conjecture}
\newtheorem{question}[theorem]{Question}
\newtheorem{example}{Example}[theorem]

\theoremstyle{definition}
\newtheorem{definition}[theorem]{Definition}

\theoremstyle{remark}
\newtheorem{remark}[theorem]{Remark}

\newcommand{\mb}{\mathbb}
\newcommand{\tb}{\textbf}
\newcommand{\tn}{\textnormal}
\newcommand{\id}{\tn{id}}

\newcommand{\E}{\mb{E}}

\newcommand{\bs}{\backslash}
\newcommand{\se}{\subseteq}
\newcommand{\tc}{\textcolor}

\title{Positivity of permutation pattern character polynomials}
\author{Christian Gaetz\footnote{C.G. is supported by a Klarman Postdoctoral Fellowship at Cornell University, and was also supported by an NSF Postdoctoral Research Fellowship under grant no. DMS-2103121.} \\ Cornell University \\ \href{mailto:crgaetz@gmail.com}{crgaetz@gmail.com}  \and 
Laura Pierson \\ Harvard University \\ \href{mailto:lcpierson73@gmail.com}{lcpierson73@gmail.com}}
\date{\today}

\begin{document}

\maketitle

\begin{abstract}
Let $N_\sigma(\pi)$ denote the number of occurrences of a permutation pattern $\sigma\in S_k$ in a permutation $\pi\in S_n$. Gaetz and Ryba \cite{GaRy21} showed using partition algebras that the $d$-th moment $M_{\sigma,d,n}(\pi)$ of $N_\sigma$ on the conjugacy class of $\pi$ is given by a polynomial in $n,m_1,\dots,m_{dk}$, where $m_i$ denotes the number of $i$-cycles of $\pi$. They also showed that the coefficient $\langle \chi^{\lambda[n]}, M_{\sigma,d,n}\rangle$ agrees with a polynomial $a_{\sigma,d}^\lambda(n)$ in $n$. This work is motivated by the conjecture that when $\sigma=\id_k$ is the identity permutation, all of these coefficients are nonnegative. We directly compute closed forms for the polynomials $a_{\id_k}^{\lambda}(n)$ in the cases $\lambda=(1),(1,1),$ and $(2)$, and use this to verify the positivity conjecture for those cases by showing that the polynomials are real-rooted with all roots less than $k$. We also study the case $a_{\sigma}^{(1)}(n)$, for which we give a formula for the polynomials and their leading coefficients. 
\end{abstract}

\section{Introduction}

\subsection{Permutation pattern polynomials}
Given two permutations $\pi=\pi(1)\dots\pi(n)$ in the symmetric group $S_n$ and $\sigma = \sigma(1)\dots\sigma(k) \in S_k$, we say that $\pi$ \emph{contains the pattern} $\sigma$ if there is a sequence $i_1,\dots,i_k\in [n]$ with $i_1<\dots<i_k$ such that $\pi(i_1),\dots,\pi(i_k)$ are ordered according to $\sigma,$ i.e. $\pi(i_a)>\pi(i_b)$ if and only if $\sigma(a)>\sigma(b)$. Such a sequence $(i_1,\dots,i_k)$ is an \emph{occurrence} of $\sigma$ in $\pi$. 

Let $N_\sigma(\pi)$ denote the number of occurrences of $\sigma$ in $\pi$. If $N_{\sigma}(\pi)=0$, then $\pi$ is said to \emph{avoid} $\sigma$.  Beginning with the work of Knuth \cite{Knuth}, the study of permutation patterns and pattern avoidance has grown into a very active subfield of combinatorics (see \cite{Bona-textbook}). Permutation patterns have been found to play an important role in many settings where algebraic or geometric objects are indexed by permutations, being ubiquitous in the study of Schubert varieties, Bruhat order, and Kazhdan--Lusztig polynomials \cite{Abe-Billey}.  

The distribution of permutation pattern occurrences has been studied by several authors. For example, it was shown by Janson--Nakamura--Zeilberger \cite{Janson-Nakamura-Zeilberger} that the distribution of $N_{\sigma}$ on uniformly random permutations from $S_n$ is asymptotically normal, and by Zeilberger that the moments of this distribution are given by polynomials in $n$ \cite{Zeil}. Zeilberger also used Maple code to compute a number of these polynomials. 

More recently, efforts have been made to understand the distribution of pattern occurrences on conjugacy classes in $S_n$ using the character theory of the symmetric group. Hultman \cite{Hultman} and Gill \cite{Gill} considered the mean of $N_\sigma$ on conjugacy classes for the cases $k=2$ and $k=3$, respectively. The first author and Ryba \cite{GaRy21} used a new approach involving partition algebras to prove that all moments of $N_{\sigma}$ on conjugacy classes are polynomials in $n,m_1,\dots,m_{dk}$ (cf. Theorem~\ref{thm:M_poly}), and concluded from this that the supports of these characters stabilize as $n \to \infty$, with the coefficients of irreducible characters given by certain polynomials $a_{\sigma,d}^{\lambda}(n)$ which are the main object of study in this paper (cf. Theorem~\ref{thm:gaetz-ryba-coefficients}). 

\begin{definition}
Following the notation of \cite{GaRy21}, let $M_{\sigma,d,n}(\pi)$ be the $d$-th moment of $N_\sigma$ on the conjugacy class $C_\pi$ containing $\pi$, namely, $$M_{\sigma,d,n}(\pi) = \frac{1}{|C_\pi|}\sum_{\pi'\in C_\pi} N_\sigma^d(\pi').$$ More generally, given $d$ patterns $\sigma_1,\dots,\sigma_d$, with $\sigma_i\in S_{k_i}$, let $M_{\sigma_1,\dots,\sigma_d,n}(\pi)$ be the expected value of the product $N_{\sigma_1}\dots N_{\sigma_d}$ on the conjugacy class of $\pi$, that is, $$M_{\sigma_1,\dots,\sigma_d,n}(\pi) = \frac{1}{|C_\pi|}\sum_{\pi'\in C_\pi}N_{\sigma_1}(\pi')\dots N_{\sigma_d}(\pi').$$
\end{definition}

Since $M_{\sigma_1,\dots,\sigma_d,n}$ is a class function on $S_n$, we can expand it in the basis of irreducible symmetric group characters $\chi^{\lambda}$, where $\lambda$ is a partition of $n$. For a partition $\lambda=(\lambda_1 \geq \dots \geq \lambda_i)$ with $|\lambda|<n$ we use $\lambda[n]$ to denote the partition $(n-|\lambda|,\lambda_1,\dots,\lambda_i)$ of $n$, when this is well-defined. Theorems~\ref{thm:M_poly} and \ref{thm:gaetz-ryba-coefficients} below are extensions of the main theorem of \cite{GaRy21} and will be proven in Appendix~\ref{sec:poly_proofs}.

\begin{theorem}
[cf. Gaetz-Ryba \cite{GaRy21}, Theorem 1.1(a)]\label{thm:M_poly}
Given any permutation patterns $\sigma_1,\dots,\sigma_d$ (not necessarily distinct and not necessarily the same size) with $\sigma_i\in S_{k_i},$ $M_{\sigma_1,\dots,\sigma_d,n}$ is a polynomial in the variables $n,m_1,\dots,m_{k_1+\dots+k_d}$ of degree at most $k_1+\dots+k_d$, where $n$ has degree 1 and $m_i$ has degree $i$.
\end{theorem}

\begin{theorem}
[cf. Gaetz-Ryba \cite{GaRy21}, Theorem 1.1(b)]
\label{thm:gaetz-ryba-coefficients}
Fix patterns $\sigma_1 \in S_{k_1}, \ldots, \sigma_d \in S_{k_d}$. Then
\[
\alpha_{\sigma_1,\ldots,\sigma_d, n}^\lambda \coloneqq \langle \chi^{\lambda[n]}, M_{\sigma_1,\ldots,\sigma_d,n}\rangle
\]
agrees for all $n \geq k_1+\dots+k_d+|\lambda|$ with a polynomial $a_{\sigma_1,\dots,\sigma_d}^\lambda(n)$ in $n$ of degree at most $k_1+\dots+k_d -|\lambda|$. In particular, this coefficient is zero if $|\lambda|>k_1+\cdots+k_d$.
\end{theorem}

The polynomials $a_{\sigma_1,\dots,\sigma_d}^\lambda(n)$ will be our main object of study in this paper. For convenience, when $\sigma_1=\cdots=\sigma_d$, we write $a_{\sigma,d}^{\lambda}$ and for $a_{\sigma_1,\ldots,\sigma_d}^{\lambda}$ and we write simply $a_{\sigma}^{\lambda}$ for $a_{\sigma,1}^{\lambda}$. We also write $M_{\sigma,n}$ for $M_{\sigma,1,n}$.

\subsection{A positivity conjecture}

Let $\id_k=12\ldots k$ denote the identity permutation in $S_k$, so that $N_{\id_k}(\pi)$ counts increasing subsequences of length $k$ in $\pi$. The following surprising positivity conjecture is the main motivation of this work.

\begin{conjecture}\label{conj:pos}
For all $k,n \in \mathbb{N}$, $M_{\id_k,n}$ is a nonnegative linear combination of irreducible symmetric group characters. Furthermore, the polynomials $a_{\id_k}^\lambda(n)$ are real-rooted, with all roots less than $k$.
\end{conjecture}

Our main theorem establishes Conjecture~\ref{conj:pos} for the coefficients $\langle \chi^{\lambda[n]}, M_{\id_k,n} \rangle$ with $|\lambda| \leq 2$. This is accomplished by giving closed formulas for the coefficients in question in Section~\ref{sec:sig=id}.

\begin{theorem}\label{thm:pos}
Conjecture~\ref{conj:pos} holds for $\lambda=\emptyset, (1), (2),$ and $(1,1)$.
\end{theorem}

\begin{remark}
For all other patterns $\id_k \neq \sigma \in S_k$, some coefficient $\langle \chi^{\mu}, M_{\sigma,n} \rangle$ must be negative, since $M_{\sigma,n}(\id_n)=0$. 
\end{remark}

Conjecture~\ref{conj:pos} and Theorem~\ref{thm:pos} suggest the following natural question:

\begin{question}
Is there a natural $S_n$-module whose character is a multiple of $M_{\id_k,n}$?
\end{question}

\subsection{Outline}
This paper is organized as follows: In Section~\ref{sec:sig=id}, we compute closed forms for the polynomials $a_{\id_k}^{(1)},a_{\id_k}^{(1,1)}$, and $a_{\id_k}^{(2)}$, and verify our positivity conjecture (Conjecture \ref{conj:pos}) for these cases. In Section~\ref{sec:lam=(1)}, we study the case $a_\sigma^{(1)}(n)$ for general $\sigma$ and give formulas for the polynomials and their leading coefficients (since leading coefficients determine positivity for large $n$). In Section~\ref{sec:conclusion}, we discuss potential pathways to prove Conjecture~\ref{conj:pos} more generally, as well as several other conjectures and open questions. Finally, in Appendix ~\ref{sec:poly_proofs}, we generalize and give an alternate proof of the polynomiality results of Gaetz and Ryba \cite{GaRy21} by describing a process for computing the polynomials in each case, and in Appendix \ref{sec:technical_lemma}, we prove a key technical lemma that is used for our main formulas in Section \ref{sec:sig=id}.

\section{The cases $a_{\id_k}^\emptyset(n), a_{\id_k}^{(1)}(n),$ $a_{\id_k}^{(2)}(n)$, and $a_{\id_k}^{(1,1)}(n)$}\label{sec:sig=id}

It follows from Theorem \ref{thm:gaetz-ryba-coefficients} that for $n\ge 2k,$ the coefficients $\alpha_{\id_k,n}^\lambda = \langle \chi^{\lambda[n]}, M_{\id_k,n}\rangle$ agree with polynomials $a_{\id_k}^\lambda(n)$ in $n$ of degree at most $k - |\lambda|.$ Our goal in this section will be to compute closed forms for the polynomials $a_{\id_k}^{(1)}(n),$ $a_{\id_k}^{(2)}(n)$, and $a_{\id_k}^{(1,1)}(n)$ (Theorem \ref{thm:sig=id}), and to verify Conjecture~\ref{conj:pos} for these cases (Theorem \ref{thm:pos}). Our approach will be to express each of these polynomials in terms of certain expected values $E(n, k, r)$ (Lemmas \ref{lem:m1Cr} and \ref{lem:m2}), for which we will then compute a closed form (Lemma \ref{lem:E(n,k,r)}). To do so, we will make use of the \emph{character polynomial} formulas, which express the symmetric group characters $\chi^{\lambda[n]}(\pi)$ as polynomials dependent only on $\lambda$ and the cycle type of $\pi$:

\begin{theorem}
[see Macdonald \cite{Macdonald}]
\label{thm:char_poly}
Let $m_i(\pi)$ denote the number of $i$-cycles in $\pi$. Then provided that $n\ge |\lambda|$, the character $\chi^{\lambda[n]}$ is a polynomial in $m_1,\dots,m_{|\lambda|}$ of degree at most $|\lambda|$, where $|\lambda|=\lambda_1 + \dots+\lambda_i$ and $m_i$ has degree $i$. Specifically, we can write $$\chi^{\lambda[n]} = \sum_{|\rho|\le |\lambda|}F_\rho^\lambda\binom{m_1}{r_1}\binom{m_2}{r_2}\dots,$$ where $\rho = 1^{r_1}2^{r_2}\dots$ and $$F_\rho^\lambda = (-1)^{|\lambda|-|\rho|}\sum_\mu \chi_\rho^\mu,$$ where the sum is taken over all $\mu$ with $|\mu|=|\rho|$ such that $|\lambda|-|\rho|$ boxes can be added to $\mu$ to get $\lambda$ with no two boxes added in the same row.
\end{theorem}

The first few character polynomials, which we will make use of, are:
\begin{align}
    \chi^{(n)} &= 1, \label{eqn:chi_emptyset}\\
    \chi^{(n-1,1)} &= m_1-1, \label{eqn:chi(1)}\\
    \chi^{(n-2,2)} &= \binom{m_1}{2} + m_2 - m_1, \label{eqn:chi(2)}\\
    \chi^{(n-2,1,1)} &= \binom{m_1}{2} - m_2 - m_1 + 1.\label{eqn:chi(1,1)}
\end{align}
In the proof of Theorem \ref{thm:sig=id}, we will use the above formulas to break each inner product into a linear combination of simpler inner products, which we will then interpret in terms of expected values, and our lemmas will allow us to calculate each expected value.

\subsection{Preliminary lemmas}

\begin{definition}\label{def:E(n,k,r)}
Let $E(n,k,r)$ denote the expected value over $\pi\in S_n$ of the number of ordered pairs $(R, T),$ where $R$ is an unordered set of $r$ fixed points in $\pi$ and $T$ is an increasing subsequence of length $k$ containing all of those fixed points ($T\supseteq R$).
\end{definition}

Knowing $E(n,k,r)$ will give us formulas for our polynomials $a_{\id_k}^{(1)}(n),$ $a_{\id_k}^{(2)}(n)$, and $a_{\id_k}^{(1,1)}(n)$ because of the character polynomial formulas (\ref{eqn:chi_emptyset}), (\ref{eqn:chi(1)}), (\ref{eqn:chi(2)}), and (\ref{eqn:chi(1,1)}), together with the following two lemmas:

\begin{lemma}\label{lem:m1Cr}
The inner product of $\binom{m_1}{r}$ with $M_{\id_k,n}$ is given by $$\left\langle \binom{m_1}{r}, M_{\id_k,n}\right\rangle = \sum_{j=0}^r \frac{E(n-r+j,k,j)}{(r-j)!}.$$
\end{lemma}

\begin{proof}
First we will interpret the inner product as an expected value. By definition, the inner product of two class functions $f$ and $g$ on $S_n$ is $$\langle f, g\rangle = \frac{1}{n!} \sum_{C\text{ a conjugacy class in }S_n}|C| \cdot f(C)\overline{g(C)}.$$ So, our inner product here can be expanded as $$\frac{1}{n!}\sum_{C\text{ a conjugacy class in }S_n} |C|\cdot \binom{m_1(C)}{r}M_{\tn{id}_k,n}(C).$$ Since $M_{\id_k,n}(C)$ is the average value of $N_{\id_k,n}(\pi)$ over $\pi \in C,$ this sum is equal to $$\frac{1}{n!}\sum_{C\text{ a conjugacy class in }S_n} \left(\frac{1}{|C|}\sum_{\pi\in C}|C|\cdot \binom{m_1(C)}{r}N_{\id_k,n}(\pi)\right),$$ which simplifies to $$\frac{1}{n!}\sum_{\pi\in S_n} \binom{m_1(\pi)}{r}N_{\id_k}(\pi).$$ 

The $\binom{m_1(\pi)}{r}$ represents the number of ways to choose a set $R$ consisting of $r$ of the $m_1(\pi)$ fixed points in $\pi,$ and the $N_{\id_k}(\pi)$ represents the number of occurrences $T$ of $\id_k$ in $\pi$, or equivalently. the number of increasing subsequences of length $k.$ Thus, this sum can be interpreted as the expected value over $\pi\in S_n$ of the number of ordered pairs $(R,T)$ consisting of an unordered set $R$ of $r$ fixed points in $\pi$ together with a length $k$ increasing subsequence $T$ in $\pi$. Note that $R$ need not contain all fixed points of $\pi,$ and unlike in Definition \ref{def:E(n,k,r)}, the fixed points in $R$ need not also be in $T.$


To get the right side, we consider cases based on $j \coloneqq |R\cap T|,$ the number of chosen fixed points contained in the subsequence. That is, we break the expected value into a sum from $j=0$ to $r,$ of the expected value over $\pi \in S_n$ of the number of pairs consisting of $r$ fixed points in $\pi$ and an increasing length $k$ subsequence in $\pi$ \emph{containing exactly $j$ of those fixed points}: $$\mathbb{E}_{\pi \in S_n}(\text{total }\#\text{ of pairs }(R,T)) = \sum_{j=0}^r\mathbb{E}_{\pi \in S_n}(\#\text{ of pairs }(R,T)\text{ s.t. }|R\cap T|=j).$$ Note that by definition, $E(n,k,r)$ represents the $r=j$ term on the right side, since in that case we require all $r$ fixed points to also be in the subsequence.

If exactly $j$ of the fixed points are in the subsequence, we can imagine removing the other $r-j$ fixed points and then relabeling the remaining elements to get a permutation $\pi' \in S_{n-r+j}$ together with a set $R'$ of $j$ fixed points in $\pi'$ and an increasing subsequence $T'$ in $\pi'$ of length $k$ \emph{containing all those fixed points}. The mapping is explained via the example below.

\begin{example} \label{ex:remove_fixed_pts}
Let $n=9,$ $r = 4,$ $j=1,$ and $k=3,$ and let $\pi$ be the permutation shown below, with the boxed numbers representing the increasing subsequence $T = (2, 5, 7)$ of length $k=3,$ and the red numbers representing the $r=4$ chosen fixed points, $R = \{1, 2, 6, 9\}:$ $$\pi = \begin{pmatrix}
\tc{red}{1} & \boxed{\tc{red}{2}} & 3 & 4 & \boxed{5} & \tc{red}{6} & \boxed{7} & 8 & \tc{red}{9} \\
\tc{red}{1} & \boxed{\tc{red}{2}} & 5 & 3 & \boxed{4} & \tc{red}{6} & \boxed{7} & 8
& \tc{red}{9} \end{pmatrix}.$$ One of the red fixed points, namely 2, is also in the subsequence, so $R\cap T = \{2\}$ is a set of size $j=1.$ We will now remove the $r-j=3$ fixed points which are not contained in the subsequence, namely 1, 6, and 9:
$$\begin{pmatrix}
\boxed{\tc{red}{2}} & 3 & 4 & \boxed{5} & \boxed{7} & 8 \\
\boxed{\tc{red}{2}} & 5 & 3 & \boxed{4} & \boxed{7} & 8
\end{pmatrix}.$$ We are left with a permutation on the set $\{2,3,4,5,7,8\}$ of size $n-r+j=6$. By mapping these 6 elements to the numbers $\{1,2,3,4,5,6\}$ in increasing order ($2\mapsto 1, 3 \mapsto 2, 4 \mapsto 3, 5\mapsto 4, 7\mapsto 5, 8\mapsto 6$), we get a permutation on the set $\{1,2,3,4,5,6\}$:
$$\pi' = \begin{pmatrix}
 \boxed{\tc{red}{1}} & 2 & 3 & \boxed{4} & \boxed{5} & 6 \\
 \boxed{\tc{red}{1}} & 4 & 2 & \boxed{3} & \boxed{5} & 6
\end{pmatrix}.$$
Thus, this process results in a permutation $\pi' \in S_6 = S_{n-r+j}$, together with a subsequence $T' = (1, 4, 5)$ (the image of the original subsequence) of length $k=3,$ and a new set of chosen fixed points $R' = \{1\},$ of size $j=1.$ The key is that now all fixed points are contained in the subsequence, since we removed the ones that were not.
\end{example}

Now, by definition, the expected value over all $\pi' \in S_{n-r+j}$ of the number ordered pairs $(R',T'),$ where $R'$ is a set of $j$ fixed points in $\pi'$ and $T'$ an increasing length $k$ subsequence containing all of them, is $E(n-r+j,k,j).$ For instance, in the example above, $E(6,3,1)$ would be the expected value over $\pi'\in S_6$ of the number of pairs $(R',T')$ with $k = |T'|=3,j = |R'|=1,$ and $R'\se T'.$ 

It remains to ``add back" the $r-j$ removed fixed points to find the expected value over all $\pi \in S_n$ of the number of pairs $(R,T)$ with $R\cap T = j.$ Each permutation $\pi'$ (together with the corresponding $R'$ and $T'$) could have come from $\binom{n}{r-j}$ different permutations $\pi$ by removing fixed $r-j$ fixed points, because we could choose any subset of $r-j$ elements of $[n]$ to be the set $R\bs T$ of the $r-j$ removed fixed points that we add back, as shown in the example below.

\begin{example} \label{ex:adding_fixed_pts}
In our previous example, we could have chosen any $r-j=3$ elements of $\{1,2,\dots,9\}$ to be the 3 removed fixed points, instead of choosing $\{1,6,9\}$ as above. For instance, suppose we choose $R\bs T = \{3,4,7\}$ to be the 3 removed fixed points. Then, to get from the same $\pi',R',$ and $T'$ as in the previous example to a different $\pi,R,$ and $T,$ we first turn $\pi'$ into a corresponding permutation on the set $\{1,2,5,6,8,9\}$ of all elements of $[n]$ except 3, 4, and 7:
$$\begin{pmatrix}
 \boxed{\tc{red}{1}} & 2 & 5 & \boxed{6} & \boxed{8} & 9 \\
 \boxed{\tc{red}{1}} & 6 & 2 & \boxed{5} & \boxed{8} & 9
\end{pmatrix}.$$
Then, we insert the three removed fixed points 3, 4, and 7 to get a different original permutation $\pi$:
$$\pi = \begin{pmatrix}
 \boxed{\tc{red}{1}} & 2 & \tc{red}{3} & \tc{red}{4} & 5 & \boxed{6} & \tc{red}{7} & \boxed{8} & 9 \\
 \boxed{\tc{red}{1}} & 6 & \tc{red}{3} & \tc{red}{4} & 2 & \boxed{5} & \tc{red}{7} & \boxed{8} & 9
\end{pmatrix}.$$ In this case, we would get $R = \{1,3,4,7\}$ and $T=(1,6,8).$
\end{example}





Since there are $\binom{n}{r-j}$ ways to choose the $r-j$ elements of $R\bs T,$ each choice of $\pi', R',$ and $T'$ corresponds to $\binom{n}{r-j}$ choices of $\pi, R,$ and $T,$ so we need to multiply by $\binom{n}{r-j}$ in going from the expected value over $\pi'\in S_{n-r+j}$ to the expected value over $\pi \in S_n.$ However, we also need to multiply by $\frac{(n-r+j)!}{n!},$ since taking an expected value over $S_{n-r+j}$ involves dividing by $(n-r+j)!$ for the $(n-r+j)!$ possible values of $\pi'$, while taking an expected value over $S_n$ involves dividing by $n!$ for the $n!$ possible values of $\pi$. Thus, in total we multiply by $\binom{n}{r-j}\cdot\frac{(n-r+j)!}{n!} = \frac{1}{(r-j)!}.$ The expected value over $\pi \in S_n$ of the number of pairs $(R,S)$ with $|R\cap T| = j$ is thus $\frac{1}{(r-j)!}E(n-r+j,k,j).$ Summing over $j$ gives the claimed formula.

\end{proof}

\begin{lemma}\label{lem:m2}
The inner product of $m_2$ with $M_{\id_k,n}$ is given by $$\langle m_2, M_{\id_k,n}\rangle = \frac{1}{2}E(n-2,k,0) + \frac{1}{k}\left(1-\frac{1}{n}\right)E(n-2,k-1,0) + \frac{1}{n}E(n-1,k,1).$$
\end{lemma}

\begin{proof}
By similar reasoning to the start of the proof of Lemma \ref{lem:m1Cr}, this inner product can be rewritten as $$\frac{1}{n!}\sum_{\pi\in S_n}m_2(\pi)N_{\id_k}(\pi).$$ Since $m_2(\pi)$ is the number of 2-cycles in $\pi$ and $N_{\id_k}(\pi)$ is the number of increasing subsequences in $\pi,$
this inner product represents the expected value over $\pi \in S_n$ of the number of pairs $((i \ j), T)$ consisting of a 2-cycle $(i \ j)$ in $\pi$ and an increasing subsequence $T$ of length $k$ in $\pi$. We will now consider cases based on the overlap between the 2-cycle and $T$. It is impossible for both $i$ and $j$ to be contained in $T,$ since if $i<j$ then $\pi(i)  = j > i = \pi(j)$.

\subsubsection*{Case 1: $i,j\not\in T.$}

If $i,j\not\in T$, then we can remove $i$ and $j$ in the same manner as in Example \ref{ex:remove_fixed_pts} to get a permutation $\pi'\in S_{n-2}$ together with an increasing subsequence $T'$ in $\pi'$ of length $k.$ 

\begin{example}
Let $n = 8$ and $k = 4,$ and let $\pi$ be the permutation below, with the boxed numbers indicating the increasing subsequence $T = (2, 3, 6, 8)$ and the red numbers indicating the 2-cycle $(i \ j) = (1 \ 4)$ (which does not overlap with $T$): 
$$\pi = \begin{pmatrix}
\tc{red}{1} & \boxed{2} & \boxed{3} & \tc{red}{4} & 5 & \boxed{6} & 7 & \boxed{8} \\
\tc{red}{4} & \boxed{2} & \boxed{5} & \tc{red}{1} & 3 & \boxed{7} & 6 & \boxed{8} \end{pmatrix}.$$ Removing 1 and 4 gives
$$\begin{pmatrix}
\boxed{2} & \boxed{3}& 5 & \boxed{6} & 7 & \boxed{8} \\
\boxed{2} & \boxed{5} & 3 & \boxed{7} & 6 & \boxed{8} \end{pmatrix},$$
and then relabeling to get a permutation in $S_{n-2} = S_6$ gives $$\pi' = \begin{pmatrix}
\boxed{1} & \boxed{2} & 3 & \boxed{4} & 5 & \boxed{6}\\
\boxed{1} & \boxed{3} & 2 & \boxed{5} & 4 & \boxed{6} \end{pmatrix}.$$ The corresponding increasing subsequence in $\pi'$ is $T' = (1,2,4,6).$
\end{example}

The expected value over $\pi'\in S_{n-2}$ of the number of choices for the increasing subsequence $T'$ of length $k$ is $E(n-2,k,0)$, since $\pi'$ need not contain any chosen fixed points. Now to translate back to $\pi,$ we need to add back in the 2-cycle $(i \ j)$. Like in Example \ref{ex:adding_fixed_pts}, we can get a pair $(\pi, T)$ given any of the $\binom{n}{2}$ choices for the pair $(i \ j)$, and any choice of the pair $(\pi', T').$ Thus, we need to multiply by $\binom{n}{2}$ in going from the expected value over $\pi'\in S_{n-2}$ to the expected value over $\pi \in S_n.$ We also need to multiply by $\frac{(n-2)!}{n!}$ since the expected value over $S_n$ involves a $\frac{1}{n!}$ while the expected value over $S_{n-1}$ involves a $\frac{1}{(n-2)!}.$ Thus, in total, we multiply by $\binom{n}{2}\cdot\frac{(n-2)!}{n!}=\frac12,$ which gives the $\frac12 E(n-2,k,0)$ term.


\subsubsection*{Case 2: $i\in T$ and $j\ne i+1.$}


In this case, we can imagine removing both $i$ and $j$ and then relabeling to get a new permutation $\pi'\in S_{n-2}$ with a corresponding increasing subsequence $T'$ of length $k-1.$ (We will explain soon the reason for the restriction $j\ne i+1.$)

\begin{example}
Choose the same $\pi$ and $T$ as in the previous two examples, but now let $(i \ j) = (3 \ 5):$
$$\pi = \begin{pmatrix}
1 & \boxed{2} & \boxed{\tc{red}{3}} & 4 & \tc{red}{5} & \boxed{6} & 7 & \boxed{8} \\
4 & \boxed{2} & \boxed{\tc{red}{5}} & 1 & \tc{red}{3} & \boxed{7} & 6 & \boxed{8} \end{pmatrix}.$$
Removing 3 and 5 gives
$$\begin{pmatrix}
1 & \boxed{2} & 4 & \boxed{6} & 7 & \boxed{8} \\
4 & \boxed{2} & 1 & \boxed{7} & 6 & \boxed{8} \end{pmatrix},$$ and relabeling in the same manner as before gives
$$\pi' = \begin{pmatrix}
1 & \boxed{2} & 3 & \boxed{4} & 5 & \boxed{6} \\
3 & \boxed{2} & 1 & \boxed{5} & 4 & \boxed{6} \end{pmatrix},$$ so $\pi'$ is an element of $S_{n-2} = S_6$ and $T' = (2,4,6)$ has length $k-1 = 3.$
\end{example}

The expected number of possible increasing subsequences $T'$ of length $k-1$ over all choices of $\pi'\in S_{n-2}$ is $E(n-2,k-1,0).$ In this case, the number of choices for $\pi$ given $\pi'$ is not always the same, so we will instead compute the expected number of choices of $\pi$ over all pairs $(\pi',T').$ 

In total, there are $(n-1)^2$ choices for where to add $i$ and $j,$ since there are $n-1$ available slots, and we can imagine independently choosing slots for $i$ and $j$ and then setting $i=j+1$ if they happen to both be inserted in the same slot. Note that $i$ and $j$ are distinguishable here, since $i$ will be in $T$ while $j$ will not. The fact that we can now think of $i$ and $j$ as being added independently is the reason for the restriction $j\ne i+1.$

\begin{example}
With $\pi'$ as before, suppose we try to insert both $i$ and $j$ between the 2 and the 3. Since we cannot have $j=i+1,$ we would set $j=3$ and $i=4.$ Since $i$ is supposed to be in $T,$ this gives
$$\pi = \begin{pmatrix}
1 & \boxed{2} & \tc{red}{3} & \boxed{\tc{red}{4}} & 5 & \boxed{6} & 7 & \boxed{8} \\
5 & \boxed{2} & \tc{red}{4} & \boxed{\tc{red}{3}} & 1 & \boxed{7} & 6 & \boxed{8}
\end{pmatrix},$$ with $T=(2,4,6,8).$ This choice of $(i \ j)$ happens to work, since $T$ is indeed an increasing subsequence of $\pi.$ However, if we had instead chosen to insert $i$ after the 2 but $j$ after the 6, we would get 
$$\pi = \begin{pmatrix}
1 & \boxed{2} & \boxed{\tc{red}{3}} & 4 & \boxed{5} & 6 & \boxed{7} & \tc{red}{8} \\
4 & \boxed{2} & \boxed{\tc{red}{8}} & 1 & \boxed{6} & 5 & \boxed{7} & \tc{red}{3}
\end{pmatrix},$$ which would not work, since $T=(2,3,5,7)$ does not give an increasing subsequence in $\pi.$
\end{example}

On average, $i$ and $j$ are each equally likely to be in any of the $k$ intervals between elements of $T',$ since these intervals all have the same average size. In order to get an increasing subsequence $T$ of length $k$ containing $i$ once $i$ and $j$ are added, they would need to both be in the same interval between elements of $T',$ which happens with probability $\frac1k.$ Thus, the expected number of choices of $\pi$ and $T$ given $\pi'$ and $T'$ is $\frac{(n-1)^2}{k}.$ So, to get the expected value over $\pi\in S_n$ from the expected value over $\pi \in S_{n-1},$ we need to multiply by $\frac{(n-1)^2}{k},$ and we also need to multiply by $\frac{(n-2)!}{n!}$ since there is a $\frac{1}{(n-2)!}$ in the expected value over $S_{n-2}$ and a $\frac{1}{n!}$ in the expected value over $S_n.$ In total, we multiply by $\frac{(n-1)^2}{k}\cdot \frac{1}{n(n-1)} = \frac1k(1-\frac1n),$ which explains the $\frac1k(1-\frac1n)E(n-2,k-1,0)$ term.

\subsubsection*{Case 3: $i\in T$ and $j = i+1.$}


In this case, instead of removing $j,$ we can think of merging $i$ and $j$ into a single fixed point to create a new permutation $\pi'\in S_{n-1}$ with a corresponding subsequence $S'$ containing the fixed point $i,$ as illustrated in the example below.

\begin{example}
We will use the same values for $n, k, \pi,$ and $T$ as in the previous example, but now let $(i \ j) = (6 \ 7),$ so $i\in T:$
$$\pi = \begin{pmatrix}
1 & \boxed{2} & \boxed{3} & 4 & 5 & \boxed{\tc{red}{6}} & \tc{red}{7} & \boxed{8} \\
4 & \boxed{2} & \boxed{5} & 1 & 3 & \boxed{\tc{red}{7}} & \tc{red}{6} & \boxed{8} \end{pmatrix}.$$
In this case, instead of removing $i$ and $j,$ we merge them into a single fixed point:
$$\begin{pmatrix}
1 & \boxed{2} & \boxed{3} & 4 & 5 & \boxed{\tc{red}{6}} & \boxed{8} \\
4 & \boxed{2} & \boxed{5} & 1 & 3 & \boxed{\tc{red}{6}} & \boxed{8} \end{pmatrix}.$$
Then, we relabel the elements after that fixed point to get a permutation $\pi'\in S_{n-1} = S_7:$
$$\pi' = \begin{pmatrix}
1 & \boxed{2} & \boxed{3} & 4 & 5 & \boxed{\tc{red}{6}} & \boxed{7} \\
4 & \boxed{2} & \boxed{5} & 1 & 3 & \boxed{\tc{red}{6}} & \boxed{7} \end{pmatrix}.$$
The resulting subsequence is $T' = (2,3,6,7),$ containing the distinguished fixed point $i=6.$
\end{example}

The expected number of such subsequences $T'$ over all $\pi' \in S_{n-1}$ is $E(n-1,k,1)$, since $S'$ is required to contain a chosen fixed point $i.$ There is only one choice of $\pi$ for each $\pi',$ so this gives the $\frac{1}{n}E(n-1,k,1)$ term, where the $\frac1n$ is because the expected value over $\pi \in S_n$ involves a $\frac{1}{n!}$ while the expected value over $\pi'\in S_{n-1}$ involves a $\frac{1}{(n-1)!}.$

\end{proof}

It now remains to compute $E(n,k,r)$. We first show that it can be written as a sum as follows:

\begin{lemma}\label{lem:E(n,k,r)_sum}
For $n\ge k,$ the expected values $E(n,k,r)$ are given by $$E(n,k,r) = \frac{1}{P(n,k)}\sum_{n_1+\dots+n_{r+1}=n-r}\ \ \sum_{k_1+\dots+k_{r+1}=k-r}\ \ \prod_{i=1}^{r+1} \binom{n_i}{k_i}^2.$$
\end{lemma}

\begin{proof}
For each pair of subsets $R \se T \se [n]$ with $|R| = r$ and $|T| = k,$ define a random variable $X_{R,T}(\pi)$ varying over $\pi \in S_n$ to equal 1 if $T$ in $\pi$ is an increasing subsequence and $R$ is a set of fixed points contained in $T$, and 0 otherwise. By linearity of expectations, $E(n,k,r) = \sum_{R,T} \mb{E}_{\pi \in S_n}(X_{R,T}(\pi)).$

The expected value of $X_{R,T}(\pi)$ represents probability over $\pi \in S_n$ that $T$ forms an increasing subsequence in $\pi$ and all elements of $R$ are fixed points of $\pi$. To compute this probability, we will ignore where $\pi$ sends the elements outside $T,$ and only consider where $\pi$ maps the elements of $T.$ The total number of ways to choose where the $k$ elements of $T$ could map under $\pi$ is $P(n,k)$, so that gives our denominator. It remains to find the number of ways $\pi$ could map $T$ such that all values in $R$ map to themselves and $T$ forms an increasing subsequence in $\pi.$

We will illustrate the terms of $T$ as being boxed and the fixed points in $R$ as red, as in our previous examples. Choose $k_1,k_2,\dots,k_{r+1}$ such that list of elements of $T$ in increasing order can be written as $$\boxed{k_1\tn{ terms}} \ \ \boxed{\tn{\tc{red}{fixed point}}} \ \ \boxed{k_2\tn{ terms}} \ \ \boxed{\tn{\tc{red}{fixed point}}} \ \ \dots \ \ \boxed{\tn{\tc{red}{fixed point}}} \ \ \boxed{k_{r+1}\tn{ terms}},$$ and choose $n_1,n_2,\dots,n_{r+1}$ such that the entire sequence $1,2,\dots,n$ in increasing order can be written as $$n_1\tn{ terms} \ \ \boxed{\tn{\tc{red}{fixed point}}} \ \ n_2\tn{ terms} \ \ \boxed{\tn{\tc{red}{fixed point}}} \ \  \dots \ \ \boxed{\tn{\tc{red}{fixed point}}} \ \ n_{r+1}\tn{ terms}.$$ Since there are $r$ fixed points, we must have $n_1+\dots+n_{r+1}=n-r$ and $k_1+\dots+k_{r+1}=k-r.$ 


Now suppose we fix the values of $k_1, \dots, k_{r+1}$ and $n_1, \dots, n_{r+1}$ (which means fixing $R$, since choosing $R$ is equivalent to choosing $n_1,\dots,n_{r+1}$) and we would like to count the number of ways to choose $T$ and $\pi(T)$ such that $\pi$ fixes all values in $R$ and $T$ is an increasing subsequence in $\pi.$ There are $\prod_{i=1}^{r+1} \binom{n_i}{k_i}$ ways to choose the remaining elements of $T$ in between the fixed points, because in the $i$th interval between fixed points, there are $n_i$ values to choose from and we must choose $k_i$ of them to be in the subsequence. Since we want the subsequence to be increasing, we have exactly the same number of choices for the images of these numbers under $\pi$. Thus, the number of ways to choose the rest of the subsequence $T$ and its image under $\pi$ is $\prod_{i=1}^{r+1} \binom{n_i}{k_1}^2.$

Summing over all possible choices of the $n_i$'s and $k_i$'s and dividing by $P(n,k)$ gives the desired formula.
\end{proof}

This sum formula will actually be sufficient to prove Conjecture ~\ref{conj:pos} for the case $\lambda=(1)$, but to get the closed forms and the proofs in the other cases, we will need the following closed form for $E(n,k,r)$:

\begin{lemma}\label{lem:E(n,k,r)}
For $n\ge k$, $E(n,k,r)$ has the closed form $$E(n,k,r) = \frac{2^{k-r}}{(r-1)!!(2k-r)!!}\binom{n-\frac{r}{2}}{k-r}.$$
\end{lemma}

The proof of this formula is left for Appendix \ref{app:E(n,k,r)_proof}. We will note that in the case $r=1$, we get
\begin{equation}
    E(n,k,1) = \frac{2^{k-1}(2n-1)(2n-3)\dots(2n-2k+3)}{(2k-1)!},
    \label{eqn:E(n,k,1)}
\end{equation}
which we will make use of in Theorem \ref{thm:sig=id}.

\subsection{Closed forms for the polynomials}

We can now use Lemmas \ref{lem:m1Cr}, \ref{lem:m2}, and \ref{lem:E(n,k,r)} together with the character polynomial formula (Theorem \ref{thm:char_poly}) to derive closed forms for $a_{\id_k}^\lambda(n)$ for $|\lambda| \le 2$:

\begin{theorem}\label{thm:sig=id}
We have the following closed forms for $a_{\id_k}^{\lambda}(n)$ for $\lambda=\emptyset,(1),(2),$ and $(1,1):$
\begin{align*}
    a_{\id_k}^\emptyset(n) &= \frac{1}{k!}\binom{n}{k}, \\ \\
    a_{\id_k}^{(1)}(n) &= \frac{2^{k-1}(2n-1)(2n-3)\dots(2n-2k+3)}{(2k-1)!} - \frac{1}{k!}\binom{n-1}{k-1}, \\ \\
    a_{\id_k}^{(2)}(n) &= -\frac{1}{n\cdot k!}\binom{n-2}{k-1} + \frac{1}{2(k-1)!}\binom{n-1}{k-2} \\
    &- \frac{2^{k-1}((2k-4)n+(2k-1))\cdot(2n-3)(2n-5)\dots(2n-2k+3)}{n\cdot(2k-1)!}, \\ \\
    a_{\id_k}^{(1,1)}(n) &= \frac{1}{k!}\binom{n-2}{k-2} + \frac{1}{n\cdot k!}\binom{n-2}{k-1} + \frac{1}{2(k-1)!}\binom{n-1}{k-2} \\
    &- \frac{2^{k-1}(2kn - (2k-1))\cdot(2n-3)(2n-5)\dots(2n-2k+3)}{n\cdot (2k-1)!}.
\end{align*}

Furthermore we have $a^{\lambda}_{\id_k}(n)=\alpha^{\lambda}_{\id_k,n}$ when $n\ge 0$ for $\lambda=\emptyset$, when $n\ge k$ for $\lambda=(1)$, and when $n\ge k+1$ for $\lambda=(2),(1,1).$ (For $k = 2$, the products $(2n-3)(2n-5)\dots(2n-2k+3)$ appearing in the formulas for $a_{\id_k}^{(2)}(n)$ and $a_{\id_k}^{(1,1)}(n)$ should be interpreted as equaling 1.)
\end{theorem}

\begin{proof}
We will consider each case separately.

\subsubsection*{Case $\lambda=\emptyset$:}

This formula is already known, and can be seen directly by noting that each of the $\binom{n}{k}$ subsequences of length $k$ in $\pi$ is equally likely to be in any order, and thus has a $\frac{1}{k!}$ chance of being increasing. This argument is valid for any $n\ge 0.$ (In the cases where $n<k$, both the polynomial and the coefficient are 0.)

\subsubsection*{Case $\lambda=(1):$}

By (\ref{eqn:chi(1)}), the character polynomial for $\lambda=(1)$ is $m_1-1.$ Using this and Lemma \ref{lem:m1Cr}, we get that for $n\ge 2,$
\begin{align*}
    \langle \chi^{(n-1,1)}, M_{\id_k,n}\rangle = \langle m_1,M_{\id_k,n}\rangle - \langle 1, M_{\id_k,n}\rangle = E(n,k,1) + E(n-1,k,0) - E(n,k,0).
\end{align*}
Plugging in the formula for $E(n,k,1)$ from (\ref{eqn:E(n,k,1)}) and the formula $E(n,k,0) = \frac{1}{k!}\binom{n}{k}$ and applying Pascal's identity to the last two terms shows that the desired polynomial formula holds for $n\ge k+1$.

To show that it also holds for $n=k$, the only term to which Lemma \ref{lem:E(n,k,r)} does not apply is $E(n-1,k,0)$, so we must check that the polynomial actually agrees with the expected value in this case, which it does since $E(k-1,k,0)=0$ (as there are no length $k$ subsequences in a permutation of length $k-1$), and the polynomial formula gives $\frac{1}{k!}\binom{k-1}{k}$, which is also 0. Thus, the formula for $a_{\id_k}^{(1)}(n)$ is valid for all $n\ge k.$

\subsubsection*{Case $\lambda=(2):$}

By (\ref{eqn:chi(2)}), the character polynomial for $\lambda=(2)$ is $\binom{m_1}{2}-m_1+m_2.$ Using this together with Lemmas~\ref{lem:m1Cr} and \ref{lem:m2}, we get for $n \geq 4$ that the character polynomial formula holds, and so:
\begin{align*}
    \langle \chi^{(n-2,2)}, M_{\id_k,n}\rangle &= \left\langle \binom{m_1}{2}, M_{\id_k,n}\right\rangle - \langle m_1, M_{\id_k,n}\rangle + \langle m_2, M_{\id_k,n}\rangle \\
    &= \frac{1}{2}E(n-2,k,0) + E(n-1,k,1) + E(n,k,2) \\
    &- E(n-1,k,0) - E(n,k,1) \\
    &+ \frac{1}{2}E(n-2,k,0) + \frac{1}{k}\left(1-\frac{1}{n}\right)E(n-2,k-1,0) + \frac{1}{n}E(n-1,k,1).
\end{align*}
Plugging in our formula from Theorem \ref{lem:E(n,k,r)} for $E(n,k,r)$ (using the form of $E(n,k,1)$ from (\ref{eqn:E(n,k,1)})) implies that for $n\ge k+2,$
\begin{align*}
    \langle \chi^{(n-2,2)}, M_{\id_k,n}\rangle &= \frac{1}{2\cdot k!}\binom{n-2}{k} + \frac{2^{k-1}(2n-3)(2n-5)\dots(2n-2k+1)}{(2k-1)!} + \frac{1}{2(k-1)!}\binom{n-1}{k-2} \\
    &- \frac{1}{k!}\binom{n-1}{k} - \frac{2^{k-1}(2n-1)(2n-3)\dots(2n-2k+3)}{(2k-1)!} \\
    &+ \frac{1}{2\cdot k!}\binom{n-2}{k} + \left(1-\frac{1}{n}\right)\frac{1}{k!}\binom{n-2}{k-1}+\frac{1}{n}\cdot\frac{2^{k-1}(2n-3)(2n-5)\dots(2n-2k+1)}{(2k-1)!}.
\end{align*}
We can cancel several terms using Pascal's identity: $$\left(2\cdot\frac{1}{2\cdot k!}\binom{n-2}{k}-\frac{1}{k!}\binom{n-1}{k}\right)+\frac{1}{k!}\binom{n-2}{k-1} = -\frac{1}{k!}\binom{n-2}{k-1}+\frac{1}{k!}\binom{n-2}{k-1}=0.$$ We can also combine the terms with denominator $(2k-1)!$ and factor out the common parts to get $$(n(2n-2k+1) - n(2n-1) + (2n-2k+1))\cdot\frac{2^{k-1}(2n-3)\dots(2n-2k+3)}{n\cdot(2k-1)!},$$ which simplifies to $$-\frac{2^{k-1}((2k-4)n +(2k-1))\cdot(2n-3)(2n-5)\dots(2n-2k+3)}{(2k-1)!}.$$ Putting all this together gives the claimed formula.

We now know that $\alpha_{\id_k,n}^{(2)}$ agrees with the claimed polynomial for $n\ge k+2$, so it remains to check $n=k+1$. In that case, Lemma \ref{lem:E(n,k,r)} applies to all the terms of the expansion except possibly the two $\frac{1}{2}E(n-2,k,0)$ terms, which sum to $E(n-2,k,0)$, so it suffices to show that these two terms agree with the polynomial when $n=k+1$. But that term becomes $E(k-1,k,0)=0$ when $n=k+1$, which, like in the $\lambda=(1)$ case, agrees with the polynomial $\frac{1}{k!}\binom{k-1}{k}=0$.

\subsubsection*{Case $\lambda=(1,1):$}

By (\ref{eqn:chi(1,1)}), the character polynomial for $\lambda=(1,1)$ is $\binom{m_1}{2}-m_1-m_2+1.$ Using this together with Lemmas \ref{lem:m1Cr} and \ref{lem:m2}, we get for $n\geq 4$ that the character polynomial formula holds, and so:
\begin{align*}
    \langle \chi^{(n-2,1,1)}, M_{\id_k,n}\rangle &= \left\langle \binom{m_1}{2}, M_{\id_k,n}\right\rangle - \langle m_1, M_{\id_k,n}\rangle - \langle m_2, M_{\id_k,n}\rangle + \langle 1, M_{\id_k,n}\rangle \\
    &= \frac{1}{2}E(n-2,k,0) + E(n-1,k,1) + E(n,k,2) \\
    &- E(n-1,k,0) - E(n,k,1) \\
    &- \frac{1}{2}E(n-2,k,0) - \frac{1}{k}\left(1-\frac{1}{n}\right)E(n-2,k-1,0) - \frac{1}{n}E(n-1,k,1) \\
    &+E(n,k,0).
\end{align*}
Plugging in the formulas from Lemma \ref{lem:E(n,k,r)} and (\ref{eqn:E(n,k,1)}), it follows that for $n\ge k+2,$
\begin{align*}
    \langle \chi^{(n-2,1,1)}, M_{\id_k,n}\rangle &= \frac{1}{2\cdot k!}\binom{n-2}{k} + \frac{2^{k-1}(2n-3)(2n-5)\dots(2n-2k+1)}{(2k-1)!} + \frac{1}{2(k-1)!}\binom{n-1}{k-2} \\
    &- \frac{1}{k!}\binom{n-1}{k} - \frac{2^{k-1}(2n-1)(2n-3)\dots(2n-2k+3)}{(2k-1)!} \\
    &- \frac{1}{2\cdot k!}\binom{n-2}{k} - \left(1-\frac{1}{n}\right)\frac{1}{k!}\binom{n-2}{k-1}-\frac{1}{n}\cdot\frac{2^{k-1}(2n-3)(2n-5)\dots(2n-2k+1)}{(2k-1)!} \\
    &+ \frac{1}{k!}\binom{n}{k}.
\end{align*}
We can cancel the $\pm\frac{1}{2\cdot k!}\binom{n-2}{k}$ terms, and combine several other terms using Pascal's identity to get $$-\frac{1}{k!}\binom{n-1}{k} + \frac{1}{k!}\binom{n}{k} - \frac{1}{k!}\binom{n-2}{k-1} = \frac{1}{k!}\binom{n-1}{k-1}-\frac{1}{k!}\binom{n-2}{k-1} = \frac{1}{k!}\binom{n-2}{k-2}.$$ We can also combine the terms with denominator $(2k-1)!$ to get $$(n(2n-2k+1) - n(2n-1) - (2n-2k+1))\cdot\frac{2^{k-1}(2n-3)(2n-5)\dots(2n-2k+3)}{n\cdot(2k-1)!},$$ which simplifies to $$-\frac{2^{k-1}(2kn-(2k-1))\cdot(2n-3)(2n-5)\dots(2n-2k+3)}{n\cdot(2k-1)!}.$$ Putting all this together gives the claimed formula for $a_{\id_k}^{(1,1)}(n)$ and proves that it agrees with $\alpha_{\id_k,n}^{(1,1)}$ for $n\ge k+2$, so it remains to consider $n=k+1$. In this case, the only terms to which Lemma \ref{lem:E(n,k,r)} does not apply are the $\pm \frac{1}{2}E(n-2,k,0)$ terms, and those terms cancel, so the formula is still valid for $n=k+1.$
\end{proof}

\begin{remark}
It follows from Theorem \ref{thm:gaetz-ryba-coefficients} that all four formulas in Theorem \ref{thm:sig=id} are polynomials in $n,$ but this is not obvious from the formulas for $a_{\id_k}^{(2)}(n)$ and $a_{\id_k}^{(1,1)}(n)$ since they involve $n$'s in the denominator. However, we can check that they are in fact polynomials in $n$ by showing in each case that the constant terms of the polynomials which are multiplied by $\frac1n$ cancel. For $a_{\id_k}^{(2)}(n),$ the relevant constant terms (coming from the first and third term of the formula) are $$-\frac{1}{\cdot k!}\cdot\frac{(-2)(-3)\dots(-k)}{(k-1)!} - \frac{2^{k-1}(2k-1)(-3)(-5)\dots(-2k+3)}{(2k-1)!}.$$ Collecting the negative signs and canceling common terms from the numerator and denominator in each fraction, this simplifies to $$\frac{(-1)^k}{(k-1)!} + \frac{2^{k-1}(-1)^{k-1}}{2\cdot 4 \cdot 6\cdot \dots \cdot(2k-1)} = 0.$$ Similarly, for $a_{\id_k}^{(1,1)}(n),$ the relevant constant terms (from the second and last terms of the formula) are $$\frac{1}{k!}\cdot\frac{(-2)(-3)\dots(-k)}{(k-1)!} - \frac{2^{k-1}(-2k+1)(-3)(-5)\dots(-2k+3)}{(2k-3)},$$ which is the same two terms as for $a_{\id_k}^{(2)}(n)$ but with the signs swapped, so again, the terms cancel. This shows that in both cases, the formulas do in fact give polynomials in $n,$ as expected.
\end{remark}

\subsection{Proof of the positivity conjecture for $\lambda=\emptyset,(1),(2)$, and $(1,1)$}

Now that we have the closed forms, we can use them to prove Theorem \ref{thm:pos}.

\begin{proof}[Proof of Theorem \ref{thm:pos}]
Again, we will consider each case separately. In each case, we will first prove that the polynomial has all roots real and less than $k$, so its sign never changes for $n\ge k$. 

Note next that for $n=k$, we get $\alpha_{\id_k,k}^\lambda\ge 0$ for all $\lambda$ since $M_{\id_k,k}$ has value 1 on the conjugacy class of the identity and 0 on all other conjugacy classes and is thus a scalar multiple of the character of the regular representation, implying that all characters have nonnegative coefficients. For $\lambda=\emptyset$ and $\lambda=(1)$, this is enough to imply that the coefficient is positive for all $n\ge k$: we know the coefficient agrees with the polynomial for all $n\ge k$ in those cases, and since the polynomial never switches sign for $n\ge k$, it must always be nonnegative.

For the cases $\lambda=(2)$ and $\lambda=(1,1)$, the above argument shows that the coefficient is nonnegative for $n=k$, and Theorem \ref{thm:sig=id} implies that the coefficient agrees with the polynomial for $n\ge k+1$, so for these cases we will complete the proof by showing that the leading coefficient is nonnegative, which together with the roots being less than $k$ implies the positivity for all $n\ge k+1$ and thus for all $n\ge k.$ 

\subsubsection*{Case $\lambda=\emptyset$:}

Since $a_{\id_k}^\emptyset(n)=\frac{1}{k!}\binom{n}{k}$, the $k$ roots in this case are $0,1,2,\dots,k-1,$ which are all real less than $k$. As explained above, that implies the positivity for this case.

\subsubsection*{Case $\lambda=(1):$}

\begin{lemma}\label{lem:a1_roots}
The $k-1$ roots of $a_{\id_k}^{(1)}(n)$ are $-1$ and a root between $i$ and $i+1$ for each $i=1,2,\dots,k-2.$
\end{lemma}

\begin{proof}
Plugging in $n=-1$ to the formula from Theorem \ref{thm:sig=id} gives $$\frac{2^{k-1}(-3)(-5)\dots(-2k+1)}{(2k-1)!}-\frac{1}{k!}\cdot\frac{(-2)(-4)\dots(-k)}{(k-1)!} = \frac{(-1)^{k-1}}{(k-1)!}-\frac{(-1)^{k-1}}{(k-1)!}=0.$$ For the remaining roots, note that if we plug in $n=1,2,\dots,k-1$ to the formula from Theorem \ref{thm:sig=id}, the $\frac{1}{k!}\binom{n-1}{k-1}$ term is 0, so $$a_{\id_k}^{(1)}(n) = \frac{2^{k-1}(2n-1)(2n-3)\dots(2n-2k+3)}{(2k-1)!}.$$ The right hand side has roots at $n=\frac{1}{2},\frac{3}{2},\dots,k-\frac{3}{2}$, so it has a root between $i$ and $i+1$ for each $i=1,2,\dots,k-2$. This means $a_{\id_k}^{(1)}$ switches sign between $i$ and $i+1$ for each of these values of $i$, so it must have a root in each interval $(i,i+1)$. Since it has degree $k-1$, this accounts for all its roots.
\end{proof}

As explained above, Lemma \ref{lem:a1_roots} together with $a_{\id_k}^{(1)}(k)\ge 0$ implies the positivity of $a_{\id_k}^{(1)}(n)$ for all $n\ge k$. \\
\\
\indent Alternatively, for this case we can actually prove Conjecture \ref{conj:pos} directly from Lemma \ref{lem:m1Cr} and Lemma \ref{lem:E(n,k,r)_sum}, without using the closed form for $a_{\id_k}^{(1)}(n):$ 

\begin{proof}[Alternate proof of the case $\lambda=(1)$]
By (\ref{eqn:chi(1)}) and Lemma \ref{lem:m1Cr}, $$a_{\id_k}^{(1)}(n) = \langle m_1-1,M_{\id_k,1},n\rangle = E(n-1,k,1) + E(n-1,k,0) - E(n,k,0).$$ It follows from Lemma ~\ref{lem:E(n,k,r)_sum} and Pascal's identity that for $n\ge k$, this polynomial can be written as \begin{align*}
    a_{\id_k}^{(1)}(n) &= \frac{1}{P(n,k)}\sum_{k_1+k_2=k-1} \binom{n_1}{k_1}^2\binom{n_2}{k_2}^2 + \frac{1}{k!}\binom{n-1}{k} - \frac{1}{k!}\binom{n}{k} \\
    &= \frac{1}{P(n,k)}\sum_{k_1+k_2=k-1} \binom{n_1}{k_1}^2\binom{n_2}{k_2}^2 - \frac{1}{k!}\binom{n-1}{k-1}.
\end{align*} If we fix $n_1$ and $n_2$ with $n_1+n_2=n-1$, then by Vandermonde's identity, $\sum_{k_1+k_2=k-1} \binom{n_1}{k_1}\binom{n_2}{k_2}=\binom{n-1}{k-1}$, and using that together with Cauchy-Schwarz, $$k\cdot\sum_{k_1+k_2=k-1} \binom{n_1}{k_1}^2\binom{n_2}{k_2}^2 \ge \left(\sum_{k_1+k_2=k-1} \binom{n_1}{k_1}\binom{n_2}{k_2}\right)^2 = \binom{n-1}{k-1}^2.$$ Dividing both sides by $\frac{k}{n}P(n,k)=k!\binom{n-1}{k-1}$, moving all terms to the left, and then averaging over $n_1$ from 1 to $n$ implies $a_{\id_k}^{(1)}(n)\ge 0$ for $n\ge k$.
\end{proof}

\subsubsection*{Case $\lambda=(2):$}

\begin{lemma}\label{lem:a2_roots}
The $k-2$ roots of $a_{\id_k}^{(2)}(n)$ are $-1$ and a root between $i$ and $i+1$ for each $i=1,2,\dots,k-3.$
\end{lemma}

\begin{proof}
To show that $-1$ is a root, plugging in $n=-1$ to our formula from Theorem \ref{thm:sig=id} gives $$\frac{1}{k!}\cdot\frac{(-3)\dots(-k-1)}{(k-1)!}-\frac{2^{k-1}(-3)(-5)\dots(-2k+1)}{(-1)(2k-1)!} + \frac{1}{2(k-1)!}\cdot\frac{(-2)(-3)\dots(-k+1)}{(k-2)!}.$$ This simplifies to $$\frac{(-1)^{k-1}(k+1)}{2(k-1)!} + \frac{(-1)^{k-2}}{(k-1)!}+\frac{(-1)^{k-2}(k-1)}{2(k-1)!}=0.$$ For the remaining roots, note that the two binomial coefficient terms are 0 for $n=2,\dots,k-2,$ so for each of these values of $n,$ $$a_{\id_k}^{(2)}(n) = -\frac{2^{k-1}((2k-4)n+(2k-1))\cdot(2n-3)\dots(2n-2k+3)}{n\cdot(2k-1)!}.$$ This polynomial has roots at $-\frac{2k-1}{2k-4},\frac{3}{2},\frac{5}{2},\dots,k-\frac{3}{2}.$ Thus, for each $i=2,3,\dots,k-2$, $a_{\id_k}^{(2)}(i)$ and $a_{\id_k}^{(2)}(i+1)$ have opposite signs, so $a_{\id_k}^{(2)}$ has a root in the interval $(i,i+1)$. For $i=1$, the $\binom{n-1}{k-2}$ term is 0 and we get \begin{align*}
    a_{\id_k}^{(2)}(1) &= -\frac{1}{k!}\cdot\frac{(-1)(-2)\dots(-k+1)}{(k-1)!} -\frac{2^{k-1}(4k-5)\cdot(-1)(-3)\dots(-2k+5)}{(2k-1)!} \\
    &=-\frac{(-1)^{k-1}}{k!}-\frac{(4k-5)(-1)^{k-2}}{(2k-1)(2k-3)\cdot(k-1)!}
\end{align*}
These terms have opposite signs, but the second has larger absolute value since $$k(4k-5)=4k^2-5k > 4k^2 - 8k + 3 = (2k-1)(2k-3).$$ Thus, the sign of $a_{\id_k}^{(2)}(1)$ matches the sign of the second term, which implies that there is also a sign flip between $a_{\id_k}^{(2)}(1)$ and $a_{\id_k}^{(2)}(2)$, and therefore a root of $a_{\id_k}^{(2)}$ in the interval $(1,2)$. Since we know $a_{\id_k}^{(2)}$ has $k-2$ roots total, this accounts for all of them.
\end{proof}

We know the positivity holds for $n=k$, so to prove it for $n\ge k+1$, it suffices to show that the polynomial has positive leading coefficient, since we know that all its roots are less than $k$, and that the coefficient agrees with the polynomial for all $n\ge k+1.$

\begin{lemma}\label{lem:a2_leadcoeff}
The polynomial $a_{\id_k}^{(2)}(n)$ has nonnegative leading coefficient for all $k\ge 2$.
\end{lemma}

\begin{proof}
The three terms in the formula from Theorem~\ref{thm:sig=id} all have degree $k-2$, and adding together their leading coefficients gives that the leading coefficient of $a_{\id_k}^{(2)}(n)$ is $$\frac{1}{2(k-1)!(k-2)!} - \frac{1}{(k-1)!k!} - \frac{2^{2k-2}(k-2)}{(2k-1)!}.$$ This is equal to $$\frac{1}{(2k-1)!}\left(\binom{k}{2}\binom{2k-1}{k-1}-\binom{2k-1}{k-1} - 2^{2k-2}(k-2)\right)$$ $$= \frac{k-2}{2(2k-1)!}\left((k+1)\binom{2k-1}{k-1}-2^{2k-1}\right).$$ For $k=2$, this is 0, and for $k=3$ we get $$\frac{1}{2\cdot 5!}\left(4\binom{5}{2} - 2^{5}\right)=\frac{8}{240}=\frac{1}{30}.$$ Now we use induction to show that this is positive for all $k\ge 3.$ Assume it is positive for $k-1$. The part outside the parentheses is positive, so it suffices to show that $(k+1)\binom{2k-1}{k-1}\ge 2^{2k-1}$. We have $$\frac{(k+1)\binom{2k-1}{k-1}}{k\binom{2k-3}{k-2}} = \frac{(k+1)(2k-1)(2k-2)}{k^2(k-1)}=\frac{2(k+1)(2k-1)}{k^2} = \frac{2(2k^2+k-1)}{k^2}>4 = \frac{2^{2k-1}}{2^{2k-3}},$$ so the result follows from the inductive hypothesis.
\end{proof}

\subsubsection*{Case $\lambda=(1,1):$}

\begin{lemma}\label{lem:a11_roots}
The $k-2$ roots of $a_{\id_k}^{(1,1)}(n)$ are a root between $i$ and $i+1$ for each $i=0,1,\dots,k-3.$
\end{lemma}

\begin{proof}
If we plug in $n=2,3,\dots,k-2$ to the formula from Theorem \ref{thm:sig=id}, the first three terms are 0, so
\begin{equation}\label{eqn:a_11}
a_{\id_k}^{(1,1)}(n) = \frac{-2^{k-1}(2kn-(2k-1))\cdot(2n-3)\dots(2n-2k+3)}{n\cdot(2k-1)!}.
\end{equation}
Since this polynomial has roots at $\frac{2k-1}{2k},\frac{3}{2},\dots,k-\frac{3}{2},$ $a_{\id_k}^{(1,1)}$ flips signs between $i$ and $i+1$ for each $i=2,\dots,k-3,$ so it has a root in the interval $(i,i+1)$. Now for $i=1$, the first three terms of $a_{\id_k}^{(1,1)}(1)$ sum to
$$\frac{1}{k!}\cdot\frac{(-1)(-2)\dots(-k+2)}{(k-2)!} + \frac{1}{k!}\cdot\frac{(-1)(-2)\dots(-k+1)}{(k+1)!} + 0 = \frac{(-1)^{k-2}}{k!}+\frac{(-1)^{k-1}}{k!}=0,$$ so actually, $a_{\id_k}^{(1,1)}$ also matches (\ref{eqn:a_11}) when $n=1$, meaning the root of (\ref{eqn:a_11}) at $\frac{3}{2}$ gives a sign flip between $a_{\id_k}^{(1,1)}(1)$ and $a_{\id_k}^{(1,1)}(2)$, so $a_{\id_k}^{(1,1)}$ must have a root in the interval $(1,2)$.

It remains to consider $i=0.$ Since the numerator has the root $\frac{2k-1}{2k}$ between 0 and 1, it suffices to show that at $n=0$, the sign of $a_{\id_k}^{(1,1)}$ matches the sign of the numerator of (\ref{eqn:a_11}), since then we will know that $a_{\id_k}^{(1,1)}(0)$ and $a_{\id_k}^{(1,1)}(1)$ have opposite signs and thus there is a root in the interval $(0,1).$ The sign of the numerator of (\ref{eqn:a_11}) at $n=0$ is $-(-1)^{k-1}=(-1)^k.$

Now to compute the sign of $a_{\id_k}^{(1,1)}(0)$, when we plug in $n=0$ to the formula from Theorem \ref{thm:sig=id}, for the two terms with an $n$ in the denominator we have to take the linear term in the numerator, which is the negative of the constant term times the sum of the reciprocals of the roots. We get
\begin{align*}
    a_{\id_k}^{(1,1)}(0) &= \frac{(-1)^k}{(k-1)!}\cdot\frac{k-1}{k}+\frac{(-1)^k}{(k-1)!}\left(\frac{1}{2}+\frac{1}{3}+\dots+\frac{1}{k}\right)+\frac{(-1)^k}{2(k-1)!} \\
    &-\frac{(-1)^k}{(k-1)!}\left(\frac{2k}{2k-1}+\frac{1}{1\frac{1}{2}}+\frac{1}{2\frac{1}{2}}+\dots+\frac{1}{k-1\frac{1}{2}}\right).
\end{align*}
Thus, factoring out $\frac{(-1)^k}{(k-1)!},$ to show that this has sign $(-1)^k$ it suffices to show that $$\frac{k-1}{k}+\left(\frac{1}{2}+\frac{1}{3}+\dots+\frac{1}{k}\right)+\frac{1}{2}>\frac{2k}{2k-1}+\frac{1}{1\frac{1}{2}}+\frac{1}{2\frac{1}{2}}+\dots+\frac{1}{k-1\frac{1}{2}}.$$ Since $\frac{1}{x}$ is a convex function, $$\frac{1}{2}\left(1+\frac{1}{2}\right)+\frac{1}{2}\left(\frac{1}{2}+\frac{1}{3}\right)+\dots+\frac{1}{2}\left(\frac{1}{k-2}+\frac{1}{k-1}\right) > \frac{1}{1\frac{1}{2}}+\frac{1}{2\frac{1}{2}}+\dots+\frac{1}{k-1\frac{1}{2}},$$ which simplifies to $$\frac{1}{2}+\left(\frac{1}{2}+\frac{1}{3}+\dots+\frac{1}{k-2}\right)+\frac{1}{2(k-1)} > \frac{1}{1\frac{1}{2}}+\frac{1}{2\frac{1}{2}}+\dots+\frac{1}{k-1\frac{1}{2}}.$$ For the remaining terms, we get $$\frac{k-1}{k}+\frac{1}{2(k-1)}+\frac{1}{k}= 1+\frac{1}{2k-2} > 1+\frac{1}{2k-1}=\frac{2k}{2k-1}.$$ It follows that $a_{\id_k}^{(1,1)}(0)$ has sign $(-1)^k$, which is the opposite sign from $a_{\id_k}^{(1,1)}(1)$, implying that there is a root in the interval $(0,1)$. Thus, there is a root in the interval $(i,i+1)$ for each $i=0,1,\dots,k-3$, completing the proof.
\end{proof}

As in the case $\lambda=(2)$, we already know that $a_{\id_k}^{(1,1)}(k)\ge 0$. We also know that the coefficient equals the polynomial $a_{\id_k}^{(1,1)}(n)$ for $n\ge k+1$ and that this polynomial never switches sign in this range since all the roots are smaller, so again, to complete the proof it suffices to check that the leading coefficient is nonnegative.

\begin{lemma}\label{lem:a11_leadcoeff}
The polynomial $a_{\id_k}^{(1,1)}(n)$ has nonnegative leading coefficient for all $k\ge 2$.
\end{lemma}

\begin{proof}
All four terms in our formula from Theorem \ref{thm:sig=id} have degree $n-2$, so the leading coefficient is the sum of their leading coefficients, which is $$\frac{1}{k!(k-2)!} + \frac{1}{k!(k-1)!} +\frac{1}{2(k-1)!(k-2)!} - \frac{2^{2k-2}k}{(2k-1)!}.$$ Factoring out $\frac{1}{(2k-1)!}$, we get $$\frac{1}{(2k-1)!}\left((k-1)\binom{2k-1}{k-1} + \binom{2k-1}{k-1}+\frac{k(k-1)}{2}\binom{2k-1}{k-1} - 2^{2k-2}k\right).$$ Combining the first three terms and factoring our $\frac{k}{2}$, this becomes $$\frac{k}{2(2k-1)!}\left((k+1)\binom{2k-1}{k-1}-2^{2k-1}\right).$$ This is exactly $\frac{k}{k-2}$ times the leading coefficient of $a_{\id_k}^{(2)}(n)$ which we computed in Lemma \ref{lem:a2_leadcoeff}, so it must be positive for $k>2.$ Also, for $k=2,$ the leading coefficient is $$\frac{2}{2\cdot 3!}\left(3\binom{3}{1}-2^{3}\right)=\frac{1}{6},$$ which is also positive.
\end{proof}
\end{proof}

\section{The case $a_\sigma^{(1)}(n)$}\label{sec:lam=(1)}

In this section, we generalize to the case $\sigma\ne \id_k$ for $\lambda=(1)$, $d=1$, and write down formulas for the associated polynomials $a_\sigma^{(1)}(n)$. Since we are interested in positivity, we also compute the leading coefficients. In this case, the polynomials are not always real-rooted and also do not always have roots less than $k$, so the sign is not necessarily the same for all $n\ge k$, but the leading coefficient will determine the sign for sufficiently large $n$.

\begin{proposition}\label{prop:lam=(1)}
The polynomials $a_\sigma^{(1)}(n)$ can be written as $$a_\sigma^{(1)}(n) = \frac{1}{P(n,k)}\sum_{i=1}^n\sum_{j=1}^k \binom{i-1}{j-1}\binom{i-1}{\sigma(j)-1}\binom{n-i}{k-j}\binom{n-i}{k-\sigma(j)} - \frac{1}{k!}\binom{n-1}{k-1}.$$
\end{proposition}

\begin{proof}
We can follow essentially the same logic as in the case $\sigma=\id_k$. We have
\begin{align*}
    a_\sigma^{(1)}(n) &= \E_{\pi\in S_n}(m_1(\pi)N_\sigma(\pi))-\E_{\pi\in S_n}(N_\sigma(\pi)) \\
    &= \frac{1}{n}\sum_{i=1}^n\E_{\pi\in S_n}(\#\tn{ occurrence of }\sigma\tn{ in }\pi\tn{ containing }i\mid i\tn{ fixed}) \\
    &+\frac{1}{n}\sum_{i=1}^n \E_{\pi\in S_n}(\#\tn{ occurrences of }\sigma\tn{ in }\pi\tn{ not containing }i\mid i\tn{ fixed}) \\
    &-\E_{\pi\in S_n}(N_\sigma(\pi)).
\end{align*}
The third term is $\frac{1}{k!}\binom{n}{k}$, because there are $\binom{n}{k}$ ways to choose a set of $k$ points and a $\frac{1}{k!}$ chance that each such set forms an increasing subsequence. The second term is $-\frac{1}{k!}\binom{n-1}{k},$ since it represents the expected number of length $k$ increasing subsequences in a permutation on $n-1$ values, as we can imagine removing the fixed point $i$ to get a new permutation $\pi'\in S_{n-1}$ with a new increasing subsequence of length $k$ in the same manner as before. Thus by Pascals' identity, the sum of these two terms is $-\frac{1}{k!}\binom{n-1}{k-1}.$

For the first term, we compute the probability that a sequence $i_1<i_2<\dots <i_{j-1} < i < i_{j+1}<\dots <i_k$ containing $i$ as its $j$th term is an occurrence of $\sigma$. The number of ways to choose such a sequence with $j-1$ terms before $i$ and $k-j$ terms after $i$ is $\binom{i-1}{j-1}\binom{n-i}{k-j}.$ For the sequence $\pi(i_1),\dots,\pi(i_k)$ to be ordered according to $\sigma$, we must have $\pi(i_\ell) < \pi(i) = i$ whenever $\sigma(\ell) < \sigma(j)$ (which happens for $\sigma(j)-1$ values of $\ell$), so there are $\binom{i-1}{\sigma(j)-1}$ ways to choose where these values map under $\pi$. Similarly, the $k-\sigma(j)$ values of $\ell$ for which $\sigma(\ell) > \sigma(j)$ must satisfy $\pi(i_\ell) > i$, so there are $\binom{n-i}{k-\sigma(j)}$ ways to choose where these values map. The total number of places the subsequence $i_1,\dots,i_k$ (not counting the fixed point $i$) could map to under $\pi$ is $P(n-1,k-1)=\frac{1}{n}P(n,k)$, so the probability the sequence actually represents an occurrence of $\sigma$ is $$\frac{\binom{i-1}{\sigma(j)-1}\binom{n-i}{k-\sigma(k)}}{\frac{1}{n}P(n,k)}.$$ Summing over all $j$ and plugging this back in gives the claimed formula for $a_\sigma^{(1)}(n).$
\end{proof}

\begin{proposition}
The leading coefficient of $a_\sigma^{(1)}(n)$ is $$\frac{1}{(2k-1)!}\sum_{j=1}^k \binom{j+\sigma(j)-2}{j-1}\binom{2k-j-\sigma(j)}{k-j} - \frac{1}{k!(k-1)!}.$$
\end{proposition}

\begin{proof}
Taking only the highest degree parts from each term in the sum $$P(n,k)a_{\sigma}^{(1)}(n)=\sum_{i=1}^n\sum_{j=1}^k \binom{i-1}{j-1}\binom{i-1}{\sigma(j)-1}\binom{n-i}{k-j}\binom{n-i}{k-\sigma(j)}$$ gives \begin{equation}
   \sum_{i=1}^n\sum_{j=1}^k\frac{i^{j+\sigma(j)-2}(n-i)^{2k-j-\sigma(j)}}{(j-1)!(\sigma(j)-1)!(k-j)!(k-\sigma(j))!}. \label{eqn:sig_lead_coef}
\end{equation}
If we fix $k$ but sum over $n$, we can see by taking a limit as $n\to \infty$ that the leading coefficient in the numerator matches the value of the integral $$\int_0^1 x^{j+\sigma(j)-2}(1-x)^{2k-j-\sigma(j)}dx.$$ This is an Eulerian integral of the first kind (see \cite{WhitWat}), so the solution is the Beta function $$\frac{(j+\sigma(j)-2)!(2k-j-\sigma(j))!}{(2k-1)!}.$$ Plugging this back into (\ref{eqn:sig_lead_coef}), the leading coefficient of the sum is $$\frac{1}{(2k-1)!}\sum_{j=1}^k \binom{j+\sigma(j)-2}{j-1}\binom{2k-j-\sigma(j)}{k-j}.$$ The degree of this leading term is $2k-1$ (since each term we are adding has degree $2k-2$). In our formula for $a_\sigma^{(1)}(n)$, this sum is divided by $P(n,k)$, which has leading term $n^k$, so $a_\sigma^{(1)}(n)$ has a leading term of degree $k-1$ with the same leading coefficient. The second term $-\frac{1}{k!}\binom{n-1}{k-1}$ also has degree $k-1$, and has leading coefficient $\frac{1}{k!(k-1)!}.$ It follows that the leading coefficient of $a_\sigma^{(1)}(n)$ is as claimed.
\end{proof}

We can now state the following condition for when $a_{\sigma}^{(1)}(n)$ is nonnegative:

\begin{corollary}\label{cor:lam=(1)}
The polynomial $a_\sigma^{(1)}(n)$ is positive for sufficiently large $n$ if and only if $$\sum_{j=1}^k \binom{j+\sigma(j)-2}{j-1}\binom{2k-j-\sigma(j)}{k-j} \ge \binom{2k-1}{k}.$$
\end{corollary}

It would be interesting to try to find a simpler condition on $\sigma$ for when this leading coefficient is positive, although it is not clear from this formula whether one exists.

\section{Future directions}\label{sec:conclusion}

In this section, we discuss a few potential pathways for proving Conjecture \ref{conj:pos} in general, as well as several other conjectures based on our work in this paper and our numerical computations of the polynomials $a_\sigma^\lambda(n)$ for $\sigma \le 4.$

\subsection{Roots of the polynomials}

Note that if we could prove the real-rootedness statement in Conjecture \ref{conj:pos} and also show that $a_{\id_k}^\lambda(n)$ has nonnegative leading coefficient, that would imply the positivity for $n\ge k+|\lambda|$, since we know from Theorem \ref{thm:gaetz-ryba-coefficients} that $\alpha_{\id_k,n}^\lambda=a_{\id_k}^\lambda(n)$ for those values. Then to prove Conjecture \ref{conj:pos}, it would remain to check the positivity for $k+1\le n \le k+|\lambda|-1,$ since as mentioned in the proof of Theorem \ref{thm:pos}, the positivity for $n=k$ follows automatically from $M_{\id_k,k}$ being a scalar multiple of the regular representation.

Lemmas \ref{lem:a1_roots}, \ref{lem:a2_roots}, and \ref{lem:a11_roots} also show that the roots are fairly evenly spaced for $\lambda=(1),(2),$ and $(1,1)$ (one root in each unit interval in a specific range, plus a root at $-1$ when $\lambda=(1),(2)$), so it would be interesting to know whether something similar holds more generally. Our work in Section \ref{sec:sig=id} together with direct computations of the polynomials for $k\le 4$ also suggests the following:
   
\begin{conjecture}\label{conj:n=-1}
When $\lambda$ has a single row, the polynomials $a_\sigma^\lambda(n)$ have $n=-1$ as a root.
\end{conjecture}
   
Lemmas \ref{lem:a1_roots} and \ref{lem:a2_roots} prove this for the cases $a_{\id_k}^{(1)}$ and $a_{\id_k}^{(2)}$, and we have verified it numerically for all $\sigma$ with $k\le 4$. This could potentially be proven for the case $a_\sigma^{(1)}$ for arbitrary $\sigma$ using the formula in Proposition \ref{prop:lam=(1)}. It is also true for the case $k=4$ that the polynomials $a_\sigma^{(1,1,1)}$ have $-1$ as a root (although this is not true for $a_\sigma^{(1,1)}$ or $a_\sigma^{(1,1,1,1)}$), so there might be a more general rule for a certain subset of these polynomials where $n=-1$ is a root.

\subsection{Possible approaches for a combinatorial positivity proof}

Another potential approach to proving Conjecture \ref{conj:pos} more generally is to use the interpretation of inner products with character polynomials as representing a linear combination of expected numbers of tuples consisting of an increasing subsequence and $r_i$ $m_i$-cycles for some $i$, as we did throughout this paper, but by constructing an injection rather than explicitly computing the terms. One thing that might provide some guidance for this approach is that the case $\lambda=(k)$ appears to be tight, in that the polynomials seem to be 0 in this case:
    
\begin{conjecture}\label{conj:lam=(k)}
If $\lambda=(k)$, we get $a_{\id_k}^{(k)}(n)=0.$
\end{conjecture}
    
We know that $a_{\id_k}^{(k)}$ is a constant (since its degree is $k-|\lambda|$), so Conjecture \ref{conj:lam=(k)} would follow from Conjecture \ref{conj:n=-1}, since a nonzero constant polynomial cannot have any roots and thus cannot have $n=-1$ as a root. If Conjecture \ref{conj:lam=(k)} holds, then any injective proof for Conjecture \ref{conj:pos} would give a bijection in this case.

From the character polynomial formulas, $a_{\id_k}^{(k)}$ represents the expected number of ordered pairs consisting of an increasing subsequence of length $k$ and a set of $k$ points which are mapped among themselves, minus the expected number of ordered pairs consisting of an increasing subsequence of length $k$ and a set of $k-1$ points which are mapped among themselves. If we could find a bijection between these two sets, this would give a combinatorial proof that $a_{\id_k}^{(k)}=0$, and potentially some clues as to how to more generally find an injective proof that $a_{\id_k}^{\lambda}(n)$ is nonnegative.

\subsection{Positivity for $\sigma \ne \id_k$}

Another interesting question would be whether there is a general rule for when the polynomials $a_\sigma^\lambda(n)$ are positive for other values of $\sigma$:

\begin{question}
Is there a rule for when $a_\sigma^\lambda(n)$ is nonnegative?
\end{question}

From our computations of the polynomials for $k\le 4$, neither the real-rootedness nor the property of all real roots being less than $k$ holds in general for $\sigma\ne \id_k$. Thus, in general the sign of $a_\sigma^\lambda(n)$ depends on $n$ and only stabilizes for sufficiently large $n$ (based on the sign of the leading coefficient). For the case $\lambda=(1)$, Corollary \ref{cor:lam=(1)} could potentially be used to help find a more explicit rule for when the polynomials $a_\sigma^{(1)}(n)$ are positive. One might hope that for sufficiently large $n$, the values of $a_\sigma^{(1)}(n)$ are monotonically decreasing under we weak Bruhat order, but this is not the case for $k=4$, as $\sigma=4321$ does not have minimal leading coefficient. However, the following does hold for $k\le4$, so we might hope that it also holds in general:

\begin{question}
Do the polynomials $a_\sigma^{(1)}(n)$ with nonnegative leading coefficients form an order ideal in the weak (or strong) Bruhat order? If so, is there a simple characterization of the maximal elements?
\end{question}

It could also be interesting to explore general formulas for these polynomials and try to find closed forms in other special cases, potentially by using similar methods to those in Section~\ref{sec:sig=id} and Section~\ref{sec:lam=(1)}.

\section*{Acknowledgements}
We thank the anonymous reviewer for their thorough reading of the original version of this paper and for providing a more concise proof of Lemma~\ref{lem:E(n,k,r)_sum}.

\printbibliography

\appendix

\section{Polynomiality proofs} \label{sec:poly_proofs}

We prove here Theorems~\ref{thm:M_poly} and \ref{thm:gaetz-ryba-coefficients}. This generalizes and gives an alternate proofs of Gaetz and Ryba's results in \cite{GaRy21}, which concerned the case $\sigma_1=\cdots=\sigma_d$.

\subsection{Polynomiality of $M_{\sigma_1,\dots,\sigma_d,n}$}

\begin{proof}[Proof of Theorem~\ref{thm:M_poly}]
We can interpret the value of $M_{\sigma_1,\dots,\sigma_d,n}$ on the conjugacy class $C$ as $$\E(\#\tn{ of ordered $d$-tuples }(T_1,\dots,T_d)\tn{ with $T_i$ an occurrence of }\sigma_i\tn{ in }\pi \mid \pi \in C).$$ To simplify notation, let $$T = T_1 \cup \dots \cup T_d.$$ By linearity of expectations, we can break the desired expected value into cases, where each case specifies how many distinct elements there are in the set $T\cup \pi(T)$, which elements of $\pi(T)$ map back to elements of $T$ (and thus which subsets of $T\cup \pi(T)$ form cycles), and also the relative ordering of the elements of $T\cup \pi(T).$

\begin{example}\label{ex:case}
Let $n = 15,$ let $C$ be the conjugacy class of permutations of cycle type $(5, 3, 3, 2, 2),$ and let $d = 2$ with $\sigma_1 = 43251$ and $\sigma_2 = 21345.$ Throughout our examples, we will use boxes or circles to indicate occurrences of particular permutation patterns, and colors to indicate cycles of particular sizes. In this example, the elements of the occurrence of $T_1$ of $\sigma_1$ will be boxed, the elements of the occurrence of $T_2$ of $\sigma_2$ will be circled, the 2-cycles will be red and orange, and the 3-cycles will be green and blue. With this notation, a particular choice of $\pi \in C$ (written in two-line notation) together with a choice of $T_1$ and $T_2$ could be as follows:
$$\pi = \begin{pmatrix}
\tc{red}{1} & \boxed{\tc{green}{2}} & \tc{red}{3} & \Circled{4} & \boxed{\Circled{\tc{orange}{5}}} & \boxed{\tc{orange}{6}} & \Circled{7} & \boxed{\tc{green}{8}} & \Circled{9} & \boxed{\tc{green}{10}} & \Circled{11} & \tc{cyan}{12} & \tc{cyan}{13} & \tc{cyan}{14} & 15 \\
\tc{red}{3} & \boxed{\tc{green}{8}} & \tc{red}{1} & \Circled{7} & \boxed{\Circled{\tc{orange}{6}}} & \boxed{\tc{orange}{5}} & \Circled{9} & \boxed{\tc{green}{10}} & \Circled{11} & \boxed{\tc{green}{2}} & \Circled{15} & \tc{cyan}{14} & \tc{cyan}{12} & \tc{cyan}{13} & 4
\end{pmatrix}.$$ In cycle notation, we have $$\pi = (4 \ 7 \ 9 \ 11 \ 15)(2 \ 8 \ 10)(12 \ 14 \ 13)(1 \ 3)(5 \ 6),$$ and the permutation pattern occurrences are $T_1 = (2,5,6,8,10)$ and $T_2 = (4,5,7,9,11).$

The case under which this choice of $(\pi, T_1, T_2)$ falls would specify that the third element of $T_1$ must equal the second element of $T_2,$ but that there are no other overlaps between them. It would also specify that $T_1$ is the union of a 3-cycle and a 2-cycle, ordered in the same way as above. Thus, the boxed and circled elements are required to look something like this, where the green elements form a 3-cycle and the orange elements form a 2-cycle: $$\begin{array}{ccccccccc}
\boxed{\tc{green}{*}} & \Circled{*} & \boxed{\Circled{\tc{orange}{*}}} & \boxed{\tc{orange}{*}} & \Circled{*} & \boxed{\tc{green}{*}} & \Circled{*} & \boxed{\tc{green}{*}} & \Circled{*}
\end{array}.$$
The case would also specify that the remaining four elements of $T_2$ (the black circled elements) map to each other under $\pi$ in order, but do not form a complete cycle (for $\pi,$ we have $4\mapsto 7 \mapsto 9 \mapsto 11$). Finally, it would specify that the relative ordering of all boxed and circled elements, including which ones map to each other, must be exactly the same as above. 

Notice that this implies that there is a fixed permutation $\sigma$ such that the list above forms an occurrence in $\pi$ of $\sigma.$ For our example $\pi$ above, the images of the boxed and circled elements are $$\begin{array}{ccccccccc}
\boxed{\tc{green}{8}} & \Circled{7} & \boxed{\Circled{\tc{orange}{6}}} & \boxed{\tc{orange}{5}} & \Circled{9} & \boxed{\tc{green}{10}} & \Circled{11} & \boxed{\tc{green}{2}} & \Circled{15}
\end{array}.$$
Relabeling the smallest of these as 1, the next smallest as 2, and so on, we get $$\begin{array}{ccccccccc}
\boxed{\tc{green}{5}} & \Circled{4} & \boxed{\Circled{\tc{orange}{3}}} & \boxed{\tc{orange}{2}} & \Circled{6} & \boxed{\tc{green}{7}} & \Circled{8} & \boxed{\tc{green}{1}} & \Circled{9}
\end{array}.$$ Thus, in this case, the elements of $T_1 \cup T_2$ form an occurrence in $\pi$ of 
$$\sigma = \begin{pmatrix}
1 & 2 & 3 & 4 & 5 & 6 & 7 & 8 & 9 \\
5 & 4 & 3 & 2 & 6 & 7 & 8 & 1 & 9
\end{pmatrix}.$$

Given this information about our case, the following permutation $\pi'$ from the same conjugacy class $C$ would also fall under the same case, with the occurrences $T_1'$ and $T_2'$ of $\sigma_1$ and $\sigma_2$ again shown boxed and circled, respectively:
$$\pi' = \begin{pmatrix}
\tc{cyan}{1} & \tc{cyan}{2} & \tc{cyan}{3} & \boxed{\tc{green}{4}} & \Circled{5} & \boxed{\Circled{\tc{orange}{6}}} & \boxed{\tc{orange}{7}} & \Circled{8} & \tc{red}{9} & \tc{red}{10} & \boxed{\tc{green}{11}} & \Circled{12} & \boxed{\tc{green}{13}} & \Circled{14} & 15 \\
\tc{cyan}{3} & \tc{cyan}{1} & \tc{cyan}{2} & \boxed{\tc{green}{11}} & \Circled{8} & \boxed{\Circled{\tc{orange}{7}}} & \boxed{\tc{orange}{6}} & \Circled{12} & \tc{red}{10} & \tc{red}{9} & \boxed{\tc{green}{13}} & \Circled{14} & \boxed{\tc{green}{4}} & \Circled{15} & 5
\end{pmatrix}.$$ 
Note that the order of the elements outside $T \cup \pi(T)$ is not specified by the case and may change.
\end{example}

For each case, we can define a corresponding ``supercase," which includes all possible choices of $T_1$ and $T_2$ subject to the same restrictions about which elements form cycles and which elements map to each other, but without any restrictions on the ordering, including no requirement that $T_1$ and $T_2$ be occurrences of $\sigma_1$ and $\sigma_2.$ To make counting a bit easier, the supercase will also specify an ordering on the cycles of each size which are contained in $T$, and on the elements of each such cycle.

\begin{example}\label{ex:supercase}
The following permutation $\pi''$, together with the indicated choice of $T_1''$ and $T_2''$, would fall under the same supercase as $\pi$ and $\pi'$ from the previous example, even though $T_1''$ and $T_2''$ are not occurrences of $\sigma_1$ or $\sigma_2$, and the green, orange, boxed, and circled elements are not ordered in the same way as before:
$$\pi'' = \begin{pmatrix}
\boxed{\tc{green}{1}} & \boxed{\tc{green}{2}} & \boxed{\tc{green}{3}} & \Circled{4} & \Circled{5} & \Circled{6} & \Circled{7} & \tc{cyan}{8} & \tc{cyan}{9} & \tc{cyan}{10} & \boxed{\tc{orange}{11}} & \boxed{\Circled{\tc{orange}{12}}} & 13 & \tc{red}{14} & \tc{red}{15} \\
\boxed{\tc{green}{3}} & \boxed{\tc{green}{1}} & \boxed{\tc{green}{2}} & \Circled{7} & \Circled{6} & \Circled{13} & \Circled{5} & \tc{cyan}{9} & \tc{cyan}{10} & \tc{cyan}{8} & \boxed{\tc{orange}{12}} & \boxed{\Circled{\tc{orange}{11}}} & 4 & \tc{red}{15} & \tc{red}{14}
\end{pmatrix}.$$
The relevant information preserved from the original case is that the boxed elements are a union of a 2-cycle and a 3-cycle, one element of that 2-cycle is also circled, and the remaining circled elements map to each other but do not form a complete cycle. We would also choose a ``first" element of the 2-cycle $(11 \ 12),$ which must be 12 since that is on the one which is also circled, and we would arbitrarily choose a ``first" element of the 3-cycle, say the 3.
\end{example}

The first key observation is that as $\pi$ ranges over the conjugacy class $C$ and as $T_1,\dots, T_d$ range over a given \emph{supercase}, the probability that $(\pi,T_1,\dots, T_d)$ actually falls into a particular \emph{case} is simply $1/|T\cup \pi(T)|!,$ because as $\pi$ ranges over $C,$ there is no reason for the elements of $T\cup \pi(T)$ to be in any particular order, and thus all orderings of them are equally likely, while each case specifies one exact ordering. The second is that given any $\pi \in C,$ the total number of choices of $T_1,\dots,T_d$ falling under a given supercase is always the same, since this number depends only on the cycle type of $\pi.$ Thus, our overall probability is $$\sum_{\text{supercases}}(\text{\# of choices of the $T_i$'s under the supercase (given $\pi$)})\cdot\frac{\#\text{ of cases under the supercase}}{|T\cup \pi(T)|!}.$$

\begin{example}
For the permutation $\pi$ from Example \ref{ex:case}, we have $$T\cup \pi(T) = \{2,4,5,6,7,8,9,10,11,15\}.$$ This is a set of size 10, so it takes on $10!$ equally likely orderings as $\pi$ ranges over $C$ and $T_1$ and $T_2$ range over the supercase specified in Example \ref{ex:supercase}. We will treat the ordering occurring in $\pi$ as the identity permutation. For the permutation $\pi''$ in Example \ref{ex:supercase}, if we write elements of $T''\cup \pi''(T)$ in the order corresponding to how they are ordered in $\pi$ (including writing the ``first" element of each cycle first, and the elements of the chain in the order in which they map to each other), we get 
$$\begin{array}{ccccccccccc}
\text{Elements in $T\cup \pi(T)$:} & \boxed{\tc{green}{2}} & \Circled{4} & \boxed{\Circled{\tc{orange}{5}}} & \boxed{\tc{orange}{6}} & \Circled{7} & \boxed{\tc{green}{8}} & \Circled{9} & \boxed{\tc{green}{10}} & \Circled{11} & 15 \\
\text{Corresponding elements in $T''\cup \pi''(T''):$} & \boxed{\tc{green}{3}} & \Circled{4} & \boxed{\Circled{\tc{orange}{12}}} & \boxed{\tc{orange}{11}} & \Circled{7} & \boxed{\tc{green}{2}} & \Circled{5} & \boxed{\tc{green}{1}} & \Circled{6} & 13
\end{array}.$$ (Note that we include the element $\pi''(6) = 13$ corresponding to $\pi(11)=15,$ since it is in $\pi''(T'')$ but not in $T''$.) Relabeling the smallest element of $T''\cup \pi''(T'')$ as 1, the second smallest as 2, and so on, we see that the ordering of $T''\cup \pi''(T'')$ in comparison to the ordering of the corresponding elements in $T\cup \pi(T)$ corresponds to the following permutation in $S_{10}$:  
$$\begin{pmatrix}
1 & 2 & 3 & 4 & 5 & 6 & 7 & 8 & 9 & 10 \\
3 & 4 & 9 & 8 & 7 & 2 & 5 & 1 & 6 & 10
\end{pmatrix}.$$
Note that if we had used $T'\cup \pi'(T')$ in place of $T''\cup \pi''(T'')$, we would have just gotten the identity permutation. Thus, only one of the $10!$ permutations would lead to $(\pi, T_1, T_2)$ falling under the case from Example \ref{ex:case}, so the probability of falling under that case would be $\frac{1}{10!}.$ The total number of valid cases falling under the supercase would be the number of orderings such that $T_1$ actually forms an occurrence of $\sigma_1$ and $T_2$ actually forms an occurrence of $\sigma_2,$ which is some constant less than $10!.$
\end{example}

It remains to show that the number of choices of the $T_i$'s given $\pi$ and a particular supercase is a polynomial in $n.$ Given a supercase, we can write $T$ as a disjoint union of \emph{cycles} $(a_1 \ a_2 \ \dots \ a_i)$, where $a_{j+1} = \pi(a_j)$ for $i=1,2,\dots,j-1$ and $a_1=\pi(a_i)$, and \emph{chains} $(b_1,b_2,\dots, b_i)$, where $b_{j+1}=\pi(b_j)$ for $j=1,2,\dots,i-1$ but $\pi^{-1}(b_1)\not\in T$ and $\pi(b_i)\not\in T$. 

\begin{example}
For $(\pi, T_1,T_2)$ from Example \ref{ex:case}, the relevant cycles would be $(2 \ 8 \ 10)$ and $(5 \ 6),$ and the relevant chain would be $(4,7,9,11).$
\end{example}

Then the number of ways to choose all the $a_i$'s (without regard to their being in the correct order specified by our case) is just $$\prod_{i\ge 1} i^{r_i}P(m_i,r_i),$$ where $r_i$ is the number of $i$-cycles in $T$ and $P(n,k)\coloneqq n(n-1)\dots(n-k+1),$ since there are $P(m_i,r_i)$ ways to choose which $i$-cycles get used, and $i$ ways to rotate each $i$-cycle. This is a polynomial whose degree equals the total number of elements of $T$ contained in the cycles, since $m_i$ has degree $i$ and thus $P(m_i, r_i)$ has degree $ir_i$. Recall that we are considering the chosen cycles to be in a particular order and to each be assigned a ``first" element, since then all $|T\cup \pi(T)|!$ possible orderings discussed above will actually be treated as different.

\begin{example}
For the supercase in Example \ref{ex:supercase}, we need to choose a 2-cycle and a rotation of it, which can be done in $2m_2$ ways, since there are $m_2$ 2-cycles to choose from ($2\cdot 2$ for the conjugacy class $C$ in our example). Then we need to choose a 3-cycle and a rotation of it, which can be done in $3m_3$ ways ($3\cdot 2$ in our example). Our polynomial so far is thus $6m_2m_3,$ which has degree $2+3=5,$ since $m_2$ has degree 2 and $m_3$ has degree 3.
\end{example}

Next, we will choose the chains in $T$ (again without regard to their being in the correct order), which will be entirely determined by choosing the first element $b_1$ of each chain. At each step where we are choosing the first element of a chain, the number of choices is \begin{equation}\label{eqn:chains}
n - (\#\tn{ of values in the chosen cycles}) - (\#\tn{ of remaining values in cycles which are too short}).
\end{equation}
If we are choosing a chain of length $i$, a cycle which is ``too short" means a cycle of length at most $i$. The first term of (\ref{eqn:chains}) is a polynomial of degree 1, the second term is a constant, and the third term is a polynomial of degree at most $i$, since it is a linear combination of $m_1,\dots,m_i$, and some constant term. Thus, at each step we multiply by a polynomial whose degree is the number of elements currently being added to $T$.

\begin{example}
For the supercase in Example \ref{ex:supercase}, to count the number of ways to choose the first element of the chain, we would first subtract the 5 elements already used, and then subtract all the remaining elements in cycles of length at most 4, which in this case is equivalent to just subtracting all elements in cycles of length at most 4. Thus, this term works out to $n - m_1 - 2m_2 - 3m_3 - 4m_4,$ so the full polynomial is $$6m_2m_3(n - m_1 - 2m_2 - 3m_3 - 4m_4).$$ This last factor has degree 4 because of the $m_4$ term (corresponding to adding 4 new elements to $T$), so in total our polynomial has degree $2 + 3 + 4 = 9,$ which is precisely the number of elements of $T$: 2 from the 2-cycle, 3 from the 3-cycle, and 4 from the chain of length 4.
\end{example}

\begin{example}
Suppose we instead wanted to choose a 5-cycle and a chain of length 3. The number of ways to choose the 5-cycle would be $5m_5.$ Then, the number of ways to choose the chain would be $n - 5 - m_1 - 2m_2 - 3m_3,$ since the chain cannot be contained in a cycle of length 1, 2, or 3, and it also cannot overlap with the particular 5-cycle already chosen, even though in general it could be part of some 5-cycle. Thus, our polynomial would be $$5m_5(n - 5 - m_1 - 2m_2 - 3m_3),$$ which has degree $5+3 = 8.$
\end{example}

We may also need to subtract some terms to ensure that there are never multiple chains that overlap with each other, including the last element of one chain mapping to the first element of another chain. However, if the problem is the chains overlapping to form an $i$-cycle for some $i,$ we would need to subtract a term of the form $im_i$ (possibly combined with a constant term to account for not overlapping with elements of chosen cycles), and if the problem is the chains overlapping to form a chain of length $i,$ we subtract a term of the form $n - m_1 - 2m_2 - \dots - im_i$, again possibly with a constant term added. In either case, we are subtracting a term of degree $i,$ and $i$ must be at most the sum of the lengths of the chains, so the degree of the new term never exceeds the sum of the lengths of the chains.

\begin{example}
The issue of overlapping chains does not occur for the supercase from Example \ref{ex:supercase}, since there was only one chain being chosen, but suppose instead that we wanted to choose two chains, $(b_1, b_2)$ of length 2 and $(c_1, c_2, c_3)$ of length 3 (and no cycles). The number of ways to choose the two chains independently would be $$(n - m_1 - 2m_2)(n - m_1 - 2m_2 - 3m_3).$$ We can then list all the ways the chains might overlap, as cases we will need to subtract:
\begin{itemize}
    \item They could both be part of the same chain of length 3, either $(b_1 = c_1, b_2 = c_2, c_3)$ or $(c_1, b_1 = c_2, b_2 = c_3)$, which could happen in $2(n - m_1 - 2m_2 - 3m_3)$ ways total.
    \item They could combine to form a chain of length 4, either $(b_1, b_2=c_1, c_2, c_3)$ or $(c_1, c_2, c_3 = b_1, b_2),$ which could happen in $2(n - m_1 - 2m_2 - 3m_3 - 4m_4)$ ways.
    \item They could form a cycle of length 4, either $(b_1 \ \ b_2 = c_1 \ \ c_2 \ \ c_3)$ or $(c_1 \ \ c_2 \ \ c_3 = b_1 \ \ b_2),$ which could happen in $2\cdot 4m_4$ ways.
    \item They could form a chain of length 5, $(b_1, b_2, c_1, c_2, c_3)$ or $(c_1, c_2, c_3, b_1, b_2)$, which could happen in $2(n - m_1 - 2m_2 - 3m_3 - 4m_4 - 5m_5)$ ways.
    \item They could form a cycle of length 5, $(b_1 \ b_2 \ c_1 \ c_2 \ c_3),$ which could happen in $5m_5$ ways.
\end{itemize}
Putting all this together, our polynomial would be 
\begin{align*}
    (n - m_1 - 2m_2)(n - m_1 - 2m_2 - 3m_3) &- 2(n-m_1-2m_2 - 3m_2) \\
    &-2(n - m_1 - 2m_2 - 3m_2 - 4m_4) \\
    &-2\cdot 4m_4 \\
    &-2(n - m_1 - 2m_2 - 3m_2 - 4m_4 - 5m_5) \\
    &- 5m_5.
\end{align*}
It has degree 5, which is precisely the sum of the sizes of our two chains.
\end{example}

In general, the number of possible ways to choose $T$ in each case is a polynomial of degree at most $|T|\le |T_1|+\dots+|T_d|=k_1+\dots+k_d,$ since very term has at most this degree.

Then to find the expected number of successful choices of $(T_1, \dots, T_d)$ over $\pi \in C$ for each case, we need to multiply the total number of choices for $(T_1, \dots, T_d)$ for the corresponding supercase by the probability over all $\pi \in C$ that the chosen elements are actually ordered in the way we want according to our particular case, which, as noted above, is just $1/|T\cup \pi(T)|!.$ Summing over all cases, our expected value is a linear combination of polynomials of degree at most $k_1+\dots+k_d$, so we get a polynomial of at most the same degree.
\end{proof}

\subsection{Polynomiality of the character coefficients}

It follows from Theorem \ref{thm:M_poly} that for $n \geq 2(k_1+\cdots+k_d)$, the coefficients $$\alpha^\lambda_{\sigma_1,\dots,\sigma_r,n} = \chi_{\lambda[n]}, \langle M_{\sigma_1,\dots,\sigma_r, n}\rangle$$
agree with polynomials $a_{\sigma_1,\dots,\sigma_d}^\lambda(n)$ in $n$ of degree at most $k_1+\dots+k_d-|\lambda|$, by the same logic as in \cite{GaRy21}: the character polynomials $\chi^{\lambda[n]}$ form a basis for the space of polynomials in $m_1,\dots,m_{k_1+\dots+k_d}$ of degree at most $k_1+\dots+k_d$ (where $m_i$ has degree $i$), and so we can expand our polynomial in this basis to get coefficients which are polynomials in $n$ of the claimed degrees. Setting $n\ge 2(k_1+\dots+k_d)$ ensures that this works, since then all the relevant characters agree with the corresponding character polynomials.

However, we can also give a more direct argument for why these coefficients are polynomials in $n$, which will also show that in fact the polynomial $a_{\sigma_1,\dots,\sigma_d}^\lambda(n)$ agrees with the coefficient for $n\ge k_1+\dots+k_d+|\lambda|$, not just for $n\ge 2(k_1+\dots+k_d).$

\begin{proof}[Proof of Theorem~\ref{thm:gaetz-ryba-coefficients}]
We will begin by expanding the inner product in a similar manner to the cases computed in Section \ref{sec:sig=id}, and then we will interpret each term of the resulting sum as an expected value. We know that the character $\chi^{\lambda[n]}$ can be written using the character polynomial formula (Theorem ~\ref{thm:char_poly}, \cite{Macdonald}) as some linear combination of terms of the form $\prod_{i\ge 1} \binom{m_i}{r_i}$. Thus, by linearity of inner products, we can write $\alpha_{\sigma_1,\dots,\sigma_d,n}^\lambda$ as a linear combination of terms of the form $$\left\langle \prod_{i\ge 1}\binom{m_i}{r_i}, M_{\sigma_1,\dots,\sigma_d,n}\right\rangle.$$ By definition, the inner product of two class functions $f$ and $g$ on $S_n$ is $$\langle f, g\rangle = \frac{1}{n!} \sum_{C\text{ a conjugacy class in }S_n}|C| \cdot f(C)\overline{g(C)}.$$ In this case, the inner product can be expanded as $$\frac{1}{n!}\sum_{C\text{ a conjugacy class in }S_n} |C|\cdot \binom{m_1(C)}{r_1}\binom{m_2(C)}{r_2}\ldots \cdot M_{\sigma_1,\dots,\sigma_d,n}(C).$$ Since $M_{\sigma_1,\dots,\sigma_d,n}(C)$ is the average value of $N_{\sigma_1}(\pi)\dots N_{\sigma_d}(\pi)$ over $\pi \in C,$ so we can rewrite this sum as $$\frac{1}{n!}\sum_{C\text{ a conjugacy class in }S_n} \left(|C|\cdot \binom{m_1(C)}{r_1}\binom{m_2(C)}{r_2}\ldots\cdot \frac{1}{|C|}\sum_{\pi\in C}N_{\sigma_1}(\pi)\dots N_{\sigma_d}(\pi)\right),$$ which is equivalent to $$\frac{1}{n!}\sum_{\pi\in S_n}\binom{m_1(\pi)}{r_1}\binom{m_2(\pi)}{r_2}\ldots\cdot N_{\sigma_1}(\pi)\dots N_{\sigma_d}(\pi).$$ This term represents the expected value $$\E_{\pi \in S_n}(\#\tn{ of tuples }(R_1,R_2,\dots,T_1,\dots,T_d) \mid R_i\tn{ is the union of $r_i$ $i$-cycles, $T_i$ is an occurrence of $\sigma_i$}),$$ since $\binom{m_i(\pi)}{r_i}$ represents the number of ways to choose a set $R_i$ of $r_i$ $i$-cycles in $\pi$ and $N_{\sigma_i}(\pi)$ represent the number of ways to choose an occurrence $T_i$ of $\sigma_i$ in $\pi.$ To help simplify notation, let $$R=\bigcup_{i\ge 1}R_i, \ \ \ \ T = T_1\cup \dots \cup T_d, \ \ \ \ r = |R| = \sum_{i\ge 1}ir_i, \ \ \ \ s = |T\bs R|,$$ so $|R\cup T|=r+s.$ 

We will consider cases based on how exactly the elements of $R \cup T$ are ordered and overlap with each other, as well as the precise ordering of their images. The approach to specifying these cases will be explained via the following example.

\begin{example}\label{ex:cases}
Suppose $n=20,$ and that $R$ consists of 2 fixed points, a 2-cycle, and a 3-cycle, so $r_1=2,$ $r_2=1,$ $r_3=1,$ $r_i = 0$ for $i\ge 4,$ and $r = r_1 + 2r_2 + 3r_3 = 7.$ Let $\sigma_1 = 312$ and $\sigma_2 = 31452.$ Then the expected number of tuples $(R_1,R_2,R_3,T_1,T_2)$ over all $\pi \in S_n$ is $$\frac{1}{n!}\sum_{\pi\in S_n}\binom{m_1(\pi)}{2} m_2(\pi) m_3(\pi) \cdot N_{312}(\pi) N_{31452}(\pi).$$ One possible choice of $\pi,R,$ and $T$ would be 
$$\begin{pmatrix}
1 & \tc{red}{2} & \boxed{3} & 4 & \Circled{\tc{green}{5}} & 6 & 7 & \Circled{8} & \tc{cyan}{9} & \boxed{\tc{orange}{10}} & \tc{cyan}{11} & \Circled{\tc{cyan}{12}} & \boxed{13}  & 14 & 15 & \tc{green}{16} & \boxed{17} & 18 & \boxed{19} & 20 
\\
3 & \tc{red}{2} & \boxed{18} & 4 & \Circled{\tc{green}{16}} & 8 & 1 & \Circled{6} & \tc{cyan}{11} & \boxed{\tc{orange}{10}} & \tc{cyan}{12} & \Circled{\tc{cyan}{9}} & \boxed{19}  & 7 & 15 & \tc{green}{5} & \boxed{20} & 13 & \boxed{14} & 17 \\
\end{pmatrix},$$ with the two fixed points 2 and 10 in $R_1$ shown in red and orange, the 2-cycle $R_2 = (5 \ 16)$ shown in green, the 3-cycle $R_3 = (9 \ 11 \ 12)$ shown in blue, the occurrence $T_1=(5,8,12)$ of $\sigma_1 = 312$ shown circled, and the occurrence $T_2=(3,10,13,17,19)$ of $\sigma_2 = 31452$ shown boxed. We get that $R = \{2,5,9,10,11,12,16\}$ is the set of colored elements, $T = \{3,5,8,10,12,13,17,19\}$ is the set of boxed or circled elements, and $s = |T\bs R| = |\{3,8,13,17,19\}| = 5.$

Now to specify what case this choice of $\pi, R,$ and $T$ falls under, we would require that the ordering of the colored, boxed, and circled elements be exactly as shown above. In this case, writing out only those elements in order, we get $$\begin{array}{cccccccccccc}
\tc{red}{2} & \boxed{3} & \Circled{\tc{green}{5}} & \Circled{8} & \tc{cyan}{9} & \boxed{\tc{orange}{10}} & \tc{cyan}{11} & \Circled{\tc{cyan}{12}} & \boxed{13} & \tc{green}{16} & \boxed{17} & \boxed{19}
\end{array}.$$ So, this particular case would specify that when written in increasing order, the colored, boxed, and circled numbers follow the pattern $$\begin{array}{cccccccccccc}
\tc{red}{*} & \boxed{*} & \Circled{\tc{green}{*}} & \Circled{*} & \tc{cyan}{*} & \boxed{\tc{orange}{*}} & \tc{cyan}{*} & \Circled{\tc{cyan}{*}} & \boxed{*} & \tc{green}{*} & \boxed{*} & \boxed{*}
\end{array}.$$ Also, the images of these boxed and circled numbers under $\pi$ are $$\begin{array}{cccccccccccc}
\tc{red}{2} & \boxed{18} & \Circled{\tc{green}{16}} & \Circled{6} & \tc{cyan}{11} & \boxed{\tc{orange}{10}} & \tc{cyan}{12} & \Circled{\tc{cyan}{9}} & \boxed{19} & \tc{green}{5} & \boxed{20} & \boxed{14}
\end{array}.$$ Relabeling the smallest of these numbers as 1, the second smallest as 2, the next smallest as 3, and so on, we get
$$\begin{array}{cccccccccccc}
\tc{red}{1} & \boxed{10} & \Circled{\tc{green}{9}} & \Circled{3} & \tc{cyan}{6} & \boxed{\tc{orange}{5}} & \tc{cyan}{7} & \Circled{\tc{cyan}{4}} & \boxed{11} & \tc{green}{2} & \boxed{12} & \boxed{8}
\end{array}.$$
Thus, our sequence $R\cup T$ of colored, boxed and circled numbers is an occurrence in $\pi$ of the permutation $$\sigma = \begin{pmatrix}
1 & 2 & 3 & 4 & 5 & 6 & 7 & 8 & 9 & 10 & 11 & 12 \\
1 & 10 & 9 & 3 & 6 & 5 & 7 & 4 & 11 & 2 & 12 & 8.
\end{pmatrix}$$ The case we are in will also require that the boxed and circled numbers $R\cup T$ always be an occurrence in $\pi$ of this particular permutation $\sigma.$

For instance, the following permutation would also fall under the same case, since the colored, boxed, and circled numbers are ordered in the same way as above, and they form an occurrence in $\pi$ of the same permutation $\sigma$:
$$\begin{pmatrix}
\tc{red}{1} & \boxed{2} & \Circled{\tc{green}{3}} & \Circled{4} & \tc{cyan}{5} & 6 & 7 & 8 & \boxed{\tc{orange}{9}} & \tc{cyan}{10} & \Circled{\tc{cyan}{11}} & \boxed{12} & \tc{green}{13} & \boxed{14} & \boxed{15} & 16 & 17 & 18 & 19 & 20 
\\
\tc{red}{1} & \boxed{14} & \Circled{\tc{green}{13}} & \Circled{4} & \tc{cyan}{10} & 6 & 7 & 8 & \boxed{\tc{orange}{9}} & \tc{cyan}{11} & \Circled{\tc{cyan}{5}} & \boxed{15} & \tc{green}{3} & \boxed{16} & \boxed{12} & 14 & 17 & 18 & 19 & 20 
\end{pmatrix}.$$

\end{example}

Observe now that given the values of $r_1,r_2,r_3,\dots$ and $\sigma_1,\sigma_2,\sigma_3,\dots,$ there are a fixed number of such cases, no matter what $n$ is, since the number of colored, boxed and circled elements is bounded as $n$ grows, and those are the only elements relevant to specifying what a particular case looks like. The idea now is to find the expected value over $\pi\in S_n$ of the number of pairs $(R,T)$ falling under a particular case, and to then sum over all cases. If we can show that the expected number of such pairs is a polynomial in $n$ for every case, it will follow that the total number of expected pairs over all cases is also a polynomial in $n.$

For the purpose of computing these expected values, we will ignore what $\pi$ does to elements outside $R\cup T,$ since this is not relevant and does not impact the expected value. For each case, we will compute the expected value to be $$\frac{\tn{\# of successful choices for $R$, $T$, and their images under $\pi$}}{\tn{total \# of choices for where $R\cup T$ could map under $\pi$}}.$$ Assuming $|R\cup T| = r + s$ for the particular case we are in, the denominator will be $P(n, r+s),$ since as $\pi$ ranges over $S_n,$ there are $n(n-1)(n-2) = \dots(n-r-s+1) = P(n, r+s)$ equally likely choices for the images of any $r+s$ elements.


To compute the numerator, we will further subdivide each case into subcases based on first choosing which elements of $[n]$ are in $R.$

\begin{example} \label{ex:subcases}
Using the first permutation from Example \ref{ex:cases}, choosing a subcase would mean choosing the colored elements, so the relevant subcase would be permutations that look like
$$\begin{pmatrix}
1 & \tc{red}{2} & 3 & 4 & \Circled{\tc{green}{5}} & 6 & 7 & 8 & \tc{cyan}{9} & \boxed{\tc{orange}{10}} & \tc{cyan}{11} & \Circled{\tc{cyan}{12}} & 13  & 14 & 15 & \tc{green}{16} & 17 & 18 & 19 & 20 
\\
* & \tc{red}{2} & * & * & \Circled{\tc{green}{16}} & * & * & * & \tc{cyan}{11} & \boxed{\tc{orange}{10}} & \tc{cyan}{12} & \Circled{\tc{cyan}{9}} & *  & * & * & \tc{green}{5} & * & * & * & * \\
\end{pmatrix}.$$ Note that we already know which elements of $R$ are boxed and circled and where all elements of $R$ map under $\pi,$ because this was specified by our case described in Example \ref{ex:cases}. However, do not yet know exactly which other elements will be boxed or circled or what their images will be under $\pi,$ although we do know that the pattern of boxed and circled elements should be $$\begin{array}{cccccccccccc}
\tc{red}{*} & \boxed{*} & \Circled{\tc{green}{*}} & \Circled{*} & \tc{cyan}{*} & \boxed{\tc{orange}{*}} & \tc{cyan}{*} & \Circled{\tc{cyan}{*}} & \boxed{*} & \tc{green}{*} & \boxed{*} & \boxed{*}
\end{array},$$ and that these elements should form an occurrence in $\pi$ of the permutation $\sigma$ from Example \ref{ex:cases}. Note that in counting the number of possibilities, we will ignore where elements which are not boxed, circled, or colored map under $\pi.$ Thus, in the case of our permutation from Example \ref{ex:cases}, the relevant information would be 
$$\begin{pmatrix}
1 & \tc{red}{2} & \boxed{3} & 4 & \Circled{\tc{green}{5}} & 6 & 7 & \Circled{8} & \tc{cyan}{9} & \boxed{\tc{orange}{10}} & \tc{cyan}{11} & \Circled{\tc{cyan}{12}} & \boxed{13}  & 14 & 15 & \tc{green}{16} & \boxed{17} & 18 & \boxed{19} & 20 
\\
* & \tc{red}{2} & \boxed{18} & * & \Circled{\tc{green}{16}} & * & * & \Circled{6} & \tc{cyan}{11} & \boxed{\tc{orange}{10}} & \tc{cyan}{12} & \Circled{\tc{cyan}{9}} & \boxed{19}  & * & * & \tc{green}{5} & \boxed{20} & * & \boxed{14} & * \\
\end{pmatrix},$$

Let us now count the number of choices under this subcase for the remaining boxed and circled elements and their images (the above being one such possibility). First we choose the elements themselves. We need to choose one of 3 and 4 to be boxed, one of 6, 7, and 8 to be circled, one of 13, 14, and 15 to be boxed, and two of 17, 18, 19, and 20 to be boxed. In total, this gives $$\binom{2}{1}\binom{3}{1}\binom{3}{1}\binom{4}{2}$$ choices so far. Next, we need to choose the images of these elements. To match the ordering of the images shown above, the image of the 2nd circled element must be chosen from among 6, 7, and 8, the image of the final boxed element must be chosen from among 13, 14, and 15, and the images of the remaining three boxed elements must be chosen from among 17, 18, 19, and 20. Thus, the total number of ways to choose the images of these elements is $$\binom{3}{1}\binom{3}{1}\binom{4}{3}.$$ Once these choices are made, all the relevant information will be determined, since we do not care where the remaining elements map under $\pi.$
\end{example}

We will now generalize this example. Given a particular case (before specifying a subcase), there are numbers $j_1,j_2,\dots,j_{r+1}$ such that the list of elements of $R\cup T$ in increasing order looks like $$\boxed{j_1\tn{ terms}} \ \ \tn{\tc{red}{element of $R$}} \ \ \boxed{j_2\tn{ terms}} \ \ \tn{\tc{red}{element of $R$}} \ \ \dots \ \ \tn{\tc{red}{element of $R$}} \ \ \boxed{j_{r+1}\tn{ terms}}.$$ 

\begin{example}
Counting the number of boxed or circled elements between consecutive colored elements in Example \ref{ex:cases} (which has $r=7$ for the 7 colored elements), we find $j_1=0$ since there are no boxed or circled elements before the \tc{red}{2}, $j_2=1$ because there is one boxed element between the \tc{red}{2} and the $\Circled{\tc{green}{5}},$ $j_3 = 1$ because there is one boxed element between the $\Circled{\tc{green}{5}}$ and the \tc{cyan}{9}, $j_4=j_5=j_6=0,$ $j_7=1$ since there is one boxed element between the $\Circled{\tc{cyan}{12}}$ and the \tc{green}{16}, and $j_8=2$ since there are 2 boxed elements after the \tc{green}{16}.
\end{example}

Similarly, there are numbers $\ell_1,\ell_2,\dots,\ell_{r+1}$ so that the list of elements of $\pi(R\cup T) = R\cup\pi(T)$ in increasing order looks like $$\boxed{\ell_1\tn{ terms}} \ \ \tn{\tc{red}{element of $R$}} \ \ \boxed{\ell_2\tn{ terms}} \ \ \tn{\tc{red}{element of $R$}} \ \ \dots \ \ \tn{\tc{red}{element of $R$}} \ \ \boxed{\ell_{r+1}\tn{ terms}}.$$ 

\begin{example}
In Example \ref{ex:cases}, we get $\ell_1 = \ell_2=0$ since none of the boxed or circled elements have images less than $\Circled{\tc{green}{5}},$ $\ell_3=1$ since one circled element has image between $\Circled{\tc{green}{5}}$ and \tc{cyan}{9}, $\ell_4=\ell_5=\ell_6=0,$ $\ell_7=1$ since one boxed element has image between $\Circled{\tc{cyan}{12}}$ and \tc{green}{16}, and $\ell_8=3$ since 3 boxed elements have image greater than \tc{green}{16}.
\end{example} 

Note that the $j_i$'s and $\ell_i$'s are determined by the ordering of the elements of $R\cup T$ specified by our particular case (without regard yet to which subcase we are in), and that $$j_1+\dots+j_{r+1}=\ell_1+\dots+\ell_{r+1}=|T\bs R| = s.$$

Next, \emph{once we have specified a subcase}, there are numbers $n_1,\dots,n_{r+1}$ such that the full list $1,\dots, n$ looks like $$n_1\tn{ terms} \ \ \tn{\tc{red}{element of $R$}} \ \ n_2\tn{ terms} \ \ \tn{\tc{red}{element of $R$}} \ \ \dots \ \ \tn{\tc{red}{element of $R$}} \ \ n_{r+1}\tn{ terms},$$ and these numbers satisfy $$n_1 + \dots + n_{r+1} = n - r.$$

\begin{example}
For the subcase in Example \ref{ex:subcases}, we get $n_1=1$ (the 1), $n_2=2$ (the 3 and 4), $n_3=3$ (the 6, 7, and 8), $n_4=n_5=n_6=0,$ $n_7=3$ (the 13, 14, and 15), and $n_8=4$ (the 17, 18, 19, and 20).
\end{example}

Now we can do a calculation like the one in Example \ref{ex:subcases}. The total number of ways we could choose $T\bs R$ (the non-colored boxed and circled elements) and their images given our particular subcase is $$\prod_{i=1}^{r+1} \binom{n_i}{j_i}\cdot \prod_{i=1}^{r+1}\binom{n_i}{\ell_i},$$ since in the $i$th block we have $n_i$ elements to choose from, and we must choose $j_i$ of them to be in $T$ and $\ell_i$ of them to be in $\pi(T)$. (One can check that this matches the computation in Example \ref{ex:subcases} for that particular subcase.) Summing over all choices for the $n_i$'s (subcases) and then dividing by the denominator $P(n,r+s)$ that we determined earlier, the expected value of $\pi \in S_n$ for the number of choices of $R,T,$ and their images that work for a particular case is
\begin{equation} \label{eqn:R_T_choices}
    \frac{1}{n(n-1)\dots(n-r-s+1)}\sum_{n_1+n_2+\dots n_{r+1} = n-r} \ \ \prod_{i=1}^{r+1}\binom{n_i}{j_i}\binom{n_i}{\ell_i}.
\end{equation}
This is a rational function in $n$, and to show that it is a polynomial, it suffices to show that all the roots of the denominator, namely, $n=0,1,\dots,r+s-1$, are also roots of the numerator. This holds because the numerator counts a particular set of ways to choose two sets of $r+s$ elements from among the elements of $[n],$ and if $n<r+s,$ that would mean choosing more elements than we have available, so the numerator must be $0.$ Thus, we get a polynomial in every case, so summing over all cases, we also get a polynomial. Finally, summing over all terms of the form $$\frac{1}{n!}\sum_{\pi\in S_n}\binom{m_1(\pi)}{r_1}\binom{m_2(\pi)}{r_2}\ldots\cdot N_{\sigma_1}(\pi)\dots N_{\sigma_d}(\pi)$$ from our expanded inner product $\langle \chi^{\lambda[n]},M_{\sigma_1,\dots,\sigma_d,n}\rangle, $ we get that $\alpha_{\sigma_1,\dots,\sigma_d,n}^\lambda$ is a polynomial as well. We write $a_{\sigma_1,\ldots,\sigma_d}^{\lambda}(n)$ for this polynomial.

This argument is valid as long as we are not dividing by 0, which is the case as long as $n\ge r+s$. Since $r \le |\lambda|$ for all terms in the character polynomial for $\chi^{\lambda[n]}$ and $s \le k_1+\dots+k_d$ in all cases, it suffices to take $n \ge k_1+\dots+k_d+|\lambda|$ to ensure that $n\ge r+s,$ and therefore that $a_{\sigma_1,\dots,\sigma_d}^\lambda(n)=\alpha_{\sigma_1,\dots,\sigma_d,n}^\lambda.$
\end{proof}

\section{Proof of Lemma ~\ref{lem:E(n,k,r)}}\label{app:E(n,k,r)_proof} \label{sec:technical_lemma}

The goal of this appendix is to prove Lemma~\ref{lem:E(n,k,r)}, which states that for $n\ge k,$ $$E(n,k,r) = \frac{2^{k-r}}{(r-1)!!(2k-r)!!}\binom{n-\frac{r}{2}}{k-r}.$$ Our proof of this formula is structured as follows:
\begin{enumerate}
    \item Use generating functions to prove the statement for $r=1$ (Claim \ref{claim:E(n,k,1)}).
    \item Use induction on $n$ together with Claim \ref{claim:E(n,k,1)} to prove the statement in general (Claim \ref{claim:E_and_E'}).
\end{enumerate}

Our proof will make use of the following two generating functions:

\begin{definition}\label{def:gen_funs}
Let $F(x,y)$ and $G(x,y)$ be the generating functions
\begin{align*}
    F(x,y) &= \sum_{n,k\ge 0} \binom{n}{k}^2x^{2n}y^{2k}, \\
    G(x,y) &= \frac{1}{2}\sum_{n,k\ge 0} \binom{2n}{2k-1} x^{2n}y^{2k}.
\end{align*}
\end{definition}

From Lemma~\ref{lem:E(n,k,r)_sum}, the $E(n,k,r)$'s can be expressed in terms of coefficients of powers of $F(x,y)$ as $$\sum_{n,k\ge 0} P(n,k)E(n,k,r)x^{2n}y^{2k} = (xy)^{2r}F(x,y)^{r+1}.$$

Our second claim will give a closed form for $E(n,k,1)$, or equivalently, a relationship between these two generating functions.

\begin{claim}\label{claim:E(n,k,1)}
The formula holds for $r=1$. Equivalently,
\begin{equation}
    E(n,k,1) = \frac{2^{k-1}}{(2k-1)!!}\binom{n-\frac{1}{2}}{k-1} = \frac{2^{k-1}(2n-1)(2n-3)\dots(2n-2k+3)}{(2k-1)!} = \frac{1}{2P(n,k)}\binom{2n}{2k-1}.
\end{equation}
In terms of generating functions, this can be written as 
\begin{equation}
    x^2y^2 F(x,y)^2 = G(x,y).\label{eqn:gen_funs}
\end{equation}
\end{claim}

\begin{proof}
First we show that all these claims are equivalent. Multiplying out the factorials and canceling common terms, we get
\begin{align*}
    \frac{2^{k-1}}{(2k-1)!!}\binom{n-\frac{1}{2}}{k-1} &= \frac{2^{k-1}}{(2k-1)(2k-3)\dots3\cdot1}\cdot\frac{(n-\frac{1}{2})(n-\frac{3}{2})\dots(n-k+\frac{3}{2})}{(k-1)(k-2)\dots2\cdot1} \\
    &= \frac{2^{k-1}}{(2k-1)(2k-3)\dots3\cdot1}\cdot\frac{(2n-1)(2n-3)\dots(2n-2k+3)}{(2k-2)(2k-4)\dots4\cdot2} \\
    &= \frac{2^{k-1}(2n-1)(2n-3)\dots(2n-2k+3)}{(2k-1)!} \\
    &= \frac{2n(2n-2)\dots(2n-2k+2)}{2\cdot n(n-1)\dots(n-k+1)}\cdot\frac{(2n-1)(2n-3)\dots(2n-2k+3)}{(2k-1)!} \\
    &= \frac{1}{2P(n,k)}\binom{2n}{2k-1}.
\end{align*}
This shows that the first three expressions are equivalent. Thus, we will prove the third expression. By Lemma~\ref{lem:E(n,k,r)_sum}, the desired statement can be rewritten as $$\sum_{n_1+n_2=n-1}\ \ \sum_{k_1+k_2=k-1} \binom{n_1}{k_1}^2\binom{n_2}{k_2}^2 = \frac{1}{2}\binom{2n}{2k-1}.$$ Rewriting this in terms of the generating functions from Definition~\ref{def:gen_funs}, we can see that this is equivalent to (\ref{eqn:gen_funs}). To show (\ref{eqn:gen_funs}), we will compute closed forms for $F(x,y)$ and $G(x,y)$, and show that they match. First, we compute a closed form for $G(x,y)$ by writing it as a difference of two geometric series and then putting the resulting two fractions over a common denominator:
\begin{align*}
    G(x,y) &= \frac{y}{2}\sum_{n\ge 0} x^{2n} \sum_{k=0}^n \binom{2n}{2k-1} y^{2k-1} \\
    &= \frac{y}{2}\sum_{n\ge 0} x^{2n}\cdot\frac{(y+1)^{2n} - (y-1)^{2n}}{2} \\
    &= \frac{y}{4}\left(\frac{1}{1-x^2(y+1)^2} - \frac{1}{1-x^2(y-1)^2}\right) \\
    &= \frac{y}{4} \left(\frac{1}{1-x^2-x^2y^2 - 2x^2y} - \frac{1}{1-x^2-x^2y^2 + 2x^2y}\right) \\
    &= \frac{y}{4}\cdot\frac{4x^2y}{(1-x^2-x^2y^2)^2 - (2x^2y^2)^2} \\
    &= \frac{x^2y^2}{(1-x^2-x^2y^2)^2 - (2x^2y^2)^2}.
\end{align*}
It remains to compute a closed form for $F(x,y)$. By Pascal's identity, for all $n\ge 1$ we have $$\binom{n}{k}^2 = \left(\binom{n-1}{k-1}+\binom{n-1}{k}\right)^2 = \binom{n-1}{k-1}^2 + 2\binom{n-1}{k-1}\binom{n-1}{k} + \binom{n-1}{k}^2,$$ so we get
\begin{align*}
    F(x,y) &= 1 + \sum_{n,k\ge 1}\binom{n-1}{k-1}^2 x^{2n}y^{2k} + \sum_{n,k\ge 1} \binom{n-1}{k}^2x^{2n}y^{2k} + 2\sum_{n\ge 1,k\ge 0}\binom{n-1}{k-1}\binom{n-1}{k}x^{2n}y^{2k} \\
    &= 1 + \sum_{n,k\ge 0}\binom{n}{k}^2 x^{2n+2}y^{2k+2} + \sum_{n,k\ge 0}\binom{n}{k}^2 x^{2n+2}y^{2k} + 2\sum_{n\ge 1,k\ge 0}\binom{n-1}{k-1}\binom{n-1}{k}x^{2n}y^{2k} \\
    & = 1 + (x^2y^2 + x^2)F(x,y) + 2\sum_{n\ge 1,k\ge 0}\binom{n-1}{k-1}\binom{n-1}{k}x^{2n}y^{2k}.
\end{align*}
Rearranging gives
\begin{equation}
    (1-x^2-x^2y^2)F(x,y) = 1 + 2x^3y^2\sum_{n\ge 1,k\ge 0}\binom{n}{k-1}\binom{n}{k}x^{2n-1}y^{2k-2}.\label{eqn:F_narayana}
\end{equation}
The generating function for the Narayana numbers (see \cite{Petersen}) is known to be $$\sum_{n\ge 1,k\ge 0}\frac{1}{n}\binom{n}{k-1}\binom{n}{k}x^{2n}y^{2k-2} = \frac{1-x^2-x^2y^2+\sqrt{1-2x^2(1+y^2) + x^4(1-y^2)^2}}{2x^2y^2}.$$ Multiplying by $y^2$ and taking the derivative with respect to $x$ gives
\begin{align*}
    2y^2\sum_{n\ge 1,k\ge 0}\binom{n}{k-1}\binom{n}{k}x^{2n-1}y^{2k-2} &= -\frac{1}{x^3} + \frac{d}{dx}\sqrt{\frac{1}{4x^4}-\frac{1+y^2}{2x^2}+(1-y)^2} \\
    &= -\frac{1}{x^3} + \frac{-\frac{1}{x^5} + \frac{1+y^2}{x^3}}{2\sqrt{\frac{1}{4x^4}-\frac{1+y^2}{2x^2}+\frac{(1-y)^2}{4}}} \\
    &= -\frac{1}{x^3}\left(1 - \frac{1-x^2-x^2y^2}{\sqrt{1 - 2x^2(1+y^2) + x^4(1-y^2)^2}}\right).
\end{align*}
Plugging this into (\ref{eqn:F_narayana}) gives
$$(1-x^2-x^2y^2)F(x,y) = 1 -x^3\cdot \frac{1}{x^3}\left(1 - \frac{1-x^2-x^2y^2}{\sqrt{1 - 2x^2(1+y^2) + x^4(2-y^2)^2}}\right),$$ so we have
$$F(x,y) = \frac{1}{\sqrt{1 - 2x^2(1+y^2) + x^4(1-y^2)^2}}$$ and thus
\begin{align*}
    x^2y^2F(x,y)^2 &= \frac{x^2y^2}{1-2x^2(1+y^2) + x^4(1-y^2)^2} \\
    &= \frac{x^2y^2}{1-2x^2-2x^2y^2 + x^4y^4 - 2x^4y^2 + x^4} \\
    &= \frac{x^2y^2}{(1-x^2-x^2y^2)^2 - (2x^2y)^2} \\
    &= G(x,y),
\end{align*}
which completes the proof.
\end{proof}

Now for our third part of our proof, we will define an additional function $E'(n,k,r)$ and simultaneously prove formulas for both $E(n,k,r)$ and $E'(n,k,r)$ by induction on $n$.

\begin{definition}\label{def:E'(n,k,r)}
Let $E'(n,k,r)$ denote the sum $$E'(n,k,r) = \frac{1}{P(n,k)}\sum_{i,j\ge 1} P(i \ j)E(i,j,r-2)\cdot\frac{1}{2}\binom{2(n-i-1)}{2(k-j-1)}.$$
\end{definition}

\begin{claim}\label{claim:E_and_E'}
$E(n,k,r)$ and $E'(n,k,r)$ are given by the formulas
\begin{align*}
    E(n,k,r) &= \frac{2^{k-r}}{(r-1)!!(2k-r)!!}\binom{n-\frac{r}{2}}{k-r}, \\
    E'(n,k,r) &= \frac{2^{k-r}}{(r-1)!!(2k-r)!!}\binom{n-\frac{r}{2}}{k+1-r}.
\end{align*}
\end{claim}

\begin{proof}
We will prove both formulas simultaneously by induction on $n$, and we will also assume in the inductive step that the formulas are known to hold for smaller values of $k$ and $r$. Our base case will cover all values of $k$ and $r$ assuming that $n$ is minimal given $k$ and $r$.

We will first find recursive formulas expressing each of $E(n,k,r)$ and $E'(n,k,r)$ in terms of smaller values ((\ref{eqn:E(n,k,r)_recur}) and (\ref{eqn:E'(n,k,r)_recur})). Using our generating functions from Definition~\ref{def:gen_funs} and equation (\ref{eqn:gen_funs}) relating them,
\begin{align*}
    \sum_{n,k\ge 0}P(n,k)E(n,k,r)x^{2n}y^{2k} &= F(x,y)\sum_{n,k\ge 0}P(n,k)E(n,k,r-2)x^{2n}y^{2k} \\
    &= x^2y^2G(x,y)\sum_{n,k\ge 0}P(n,k)E(n,k,r-2)x^{2n}y^{2k}.
\end{align*}
Since $G(x,y) = \sum_{n,k\ge 0}\frac{1}{2}\binom{2n}{2k-1}x^{2n}y^{2k},$ this can be rewritten as $$P(n,k)E(n,k,r) = \sum_{i,j\ge 1} P(i \ j)E(i,j,r-2)\cdot\frac{1}{2}\binom{2(n-i-1)}{2(k-j-1)-1}.$$ Using Pascal's identity twice, 
\begin{align*}
    \binom{2n-2i-2}{2k-2j-3} &= \binom{2n-2i-3}{2k-2j-4} + \binom{2n-2i-3}{2k-2j-2} \\ & = \binom{2n-2i-4}{2k-2j-5} + 2\binom{2n-2i-4}{2k-2j-4} + \binom{2n-2i-4}{2k-2j-3}.
\end{align*}
Thus, our right hand side splits into
\begin{align*}
    P(n,k)E(n,k,r) &= \sum_{i,j\ge 1}P(i \ j)E(i,j,r-2)\cdot\frac{1}{2}\binom{2(n-i-2)}{2(k-j-2)-1} \\
    &+ \sum_{i,j\ge 1}P(i \ j)E(i,j,r-2)\cdot\frac{1}{2}\binom{2(n-i-2)}{2(k-j-2)} \\
    &+ 2\sum_{i,j\ge 1}P(i \ j)E(i,j,r-2)\cdot\frac{1}{2}\binom{2(n-i-2)}{2(k-j-1)}.
\end{align*}
From the recurrence for $E(n,k,r)$ and the definition of $E'(n,k,r)$, we can rewrite this as $$P(n, k) E(n,k, r) = P(n-1,k-1)E(n-1,k-1,r) + P(n-1,k)E(n-1,k,r) + 2P(n-1,k-1)E'(n-1,k-1,r).$$ Dividing by $P(n-1,k-1)$ on both sides, we get \begin{equation}\label{eqn:E(n,k,r)_recur}
    nE(n,k,r) = E(n-1,k-1,r) + (n-k)E(n-1,k,r) + 2E'(n-1,k-1,r).
\end{equation} We can use a similar argument to get a recurrence for $E'(n,k,r).$ We have to take out the term with a $\binom{0}{0}$ in it, since we cannot split $\binom{0}{0}$ using Pascal's identity. We get
\begin{align*}
    P(n,k)E(n,k,r) &= P(n-1,k-1)E(n-1,k-1,r-2)\cdot\frac{1}{2}\binom{0}{0} \\
    &+\sum_{i,j\ge 1}P(i \ j)E(i,j,r-2)\cdot\frac{1}{2}\binom{2(n-i-2)}{2(k-j-2)} \\
    &+ \sum_{i,j\ge 1}P(i \ j)E(i,j,r-2)\cdot\frac{1}{2}\binom{2(n-i-2)}{2(k-j-1)} \\
    &+ 2\sum_{i,j\ge 1}P(i \ j)E(i,j,r-2)\cdot\frac{1}{2}\binom{2(n-i-2)}{2(k-j-1)-1}.
\end{align*}
Applying the recurrence to each of these sums and then dividing by $P(n-1,k-1),$ we get the recurrence
\begin{equation}\label{eqn:E'(n,k,r)_recur}
    nE'(n,k,r) = \frac{1}{2}E(n-1,k-1,r-2) + E'(n-1,k-1,r) + (n-k)E'(n-1,k,r) + 2(n-k)E(n-1,k,r).
\end{equation} We can now proceed with the induction:\\
\\
\tb{Base case:} For our base case, the smallest that $n$ can be for a given $k$ to make $E(n,k,r)$ nonzero is if $n=k$, since otherwise there cannot be any increasing subsequences of length $k$. In this case, we get $E(k,k,r) = \frac{1}{k!}\binom{k}{r}.$ The only way to have an increasing subsequence of length $k$ when $n=k$ is if the permutation is the identity (which happens with probability $\frac{1}{k!}$). In that case every subset of size $r$ is a set of $r$ fixed points contained in the subsequence, and there are $\binom{k}{r}$ such subsets. To check that our formula gives the same result in this case, we can plug in $n=k$ to get
\begin{align*}
    \frac{2^{k-r}(k-\frac{r}{2})\dots(\frac{r}{2}+1)}{(r-1)!!(2k-r)!!(k-r)!} &= \frac{(2k-r)(2k-r-2)\dots(r+2)}{[(r-1)(r-3)\dots]\cdot[(2k-r)\dots(r+2)r(r-2)\dots]\cdot(k-r)!} \\
    &= \frac{1}{r!(k-r)!} = \frac{1}{k!}\binom{k}{r} = E(k,k,r).
\end{align*}
By the above calculation and (\ref{eqn:E(n,k,r)_recur}) together with Pascal's identity, we get
\begin{align*}
    E'(k,k,r) &= \frac{(k+1)E(k+1,k+1,r) - E(k,k,r)}{2} \\
    &= \frac{(k+1)\cdot\frac{1}{(k+1)!}\binom{k+1}{r} -\frac{1}{k!}\binom{k}{r}}{2}  = \frac{\frac{1}{k!}\binom{k}{r-1}}{2} = \frac{1}{2\cdot (r-1)!(k+1-r)!} \\
    &= \frac{1}{[(r-1)(r-3)\dots]\cdot[(2k-r)(2k-r-2)\dots]}\cdot\frac{(2k-r)(2k-r-2)\dots}{2\cdot[(r-2)(r-4)\dots]\cdot(k+1-r)!} \\
    &= \frac{1}{[(r-1)(r-3)\dots]\cdot[(2k-r)(2k-r-2)\dots]}\cdot\frac{2^{k-\frac{r}{2}}(k-\frac{r}{2})(k-\frac{r}{2}-1)\dots}{2\cdot[2^{\frac{r}{2}-1}(\frac{r}{2}-1)(\frac{r}{2}-2)\dots]\cdot(k+1-r)!}\\
    &= \frac{2^{k-r}}{(r-1)!!(2k-r)!!}\binom{k-\frac{r}{2}}{k+1-r}.
\end{align*}
This completes the base case. \\
\\
\tb{Inductive step:} For the inductive hypothesis, assume both formulas hold when $n$ is strictly smaller and $k$ and $r$ are weakly smaller than their current values. First we will prove that the formula for $E(n,k,r)$ holds. By the inductive hypothesis, we can plug our formulas into (\ref{eqn:E(n,k,r)_recur}) to give
\begin{align*}
    nE(n,k,r) &= E(n-1, k-1, r) + (n-k)E(n-1,k,r) + 2E'(n-1,k-1,r) \\
    &= \frac{2^{k-r-1}\binom{n-1-\frac{r}{2}}{k-1-r}}{(r-1)!!(2k-2-r)!!} + \frac{2^{k-r}(n-k)\binom{n-1-\frac{r}{2}}{k-r}}{(r-1)!!(2k-r)!!} + \frac{2^{k-r-1}\cdot2\binom{n-1-\frac{r}{2}}{k-r}}{(r-1)!!(2k-r)!!}\\
    &=\frac{2^{k-r-1}\left((2k-r)\binom{n-1-\frac{r}{2}}{k-1-r}+2(n-k)\binom{n-1-\frac{r}{2}}{k-r}+2(2k-r)\binom{n-1-\frac{r}{2}}{k-r}\right)}{(r-1)!!(2k-r)!!}.
\end{align*}
Rearranging and using Pascal's identity, we get
\begin{align*}
    nE(n,k,r) &=\frac{2^{k-r-1}(2n\binom{n-1- \frac{r}{2}}{k-r} + (2k-r)\binom{n-1-\frac{r}{2}}{k-1-r} + 2(k-r)\binom{n-1-\frac{r}{2}}{k-r})}{(r-1)!!(2k-r)!!} \\
    &= \frac{2^{k-r}\left(n\binom{n-1-\frac{r}{2}}{k-r} + (k-\frac{r}{2})\binom{n-1-\frac{r}{2}}{k-1-\frac{r}{2}}+(n-k+\frac{r}{2})\binom{n-1-\frac{r}{2}}{k-1-r}\right)}{(r-1)!!(2k-r)!!} \\
    &= \frac{2^{k-r}\left(n\binom{n-1-\frac{r}{2}}{k-r} + n\binom{n-1-\frac{r}{2}}{k-1-r}\right)}{(r-1)!!(2k-r)!!} \\
    &= \frac{2^{k-r}\cdot n}{(r-1)!!(2k-r)!!}\binom{n-\frac{r}{2}}{k-r}.
\end{align*}
Dividing by $n$ shows that the formula for $E(n,k,r)$ holds for the current value of $n$.

It remains to show prove that the formula for $E'(n,k,r)$ holds as well. Using the inductive hypothesis to plug our formulas into (\ref{eqn:E'(n,k,r)_recur}), we get
\begin{align*}
    nE'(n,k,r) &= \frac{1}{2}E(n-1,k-1,r-2) + E'(n-1,k-1,r) + (n-k)E'(n-1,k,r) + 2(n-k)E(n-1,k,r) \\
    &= \frac{\frac{1}{2}\cdot2^{k-r+1}\binom{n-\frac{r}{2}}{k+1-r}}{(r-3)!!(2k-r)!!} + \frac{2^{k-r-1}\binom{n-1-\frac{r}{2}}{k-r}}{(r-1)!!(2k-2-r)!!}+\frac{2^{k-r}(n-k)\binom{n-1-\frac{r}{2}}{k+1-r}}{(r-1)!!(2k-r)!!} + \frac{2^{k-r}\cdot2(n-k)\binom{n-1-\frac{r}{2}}{k-r}}{(r-1)!!(2k-r)!!}\\
    &=\frac{2^{k-r}\left((r-1)\binom{n-\frac{r}{2}}{k+1-r} + (k-\frac{r}{2})\binom{n-1-\frac{r}{2}}{k-r}+(n-k)\binom{n-1-\frac{r}{2}}{k+1-r}+2(n-k)\binom{n-1-\frac{r}{2}}{k-r}\right)}{(r-1)!!(2k-r)!!}.
\end{align*}
Again, rearranging and using Pascal's identity, we get
\begin{align*}
    nE'(n,k,r) &=
    \frac{2^{k-r}\left((r-1)\binom{n-\frac{r}{2}}{k+1-r} + \left((n-r+1)\binom{n-1-\frac{r}{2}}{k-r}+(n-r+1)\binom{n-1-\frac{r}{2}}{k+1-r}\right)\right)}{(r-1)!!(2k-r)!!} \\
    &- \frac{2^{k-r}\left((k+1-r)\binom{n-1-\frac{r}{2}}{k+1-r} - (n-k+\frac{r}{2}-1)\binom{n-1-\frac{r}{2}}{k-r} \right)}{(r-1)!!(2k-r)!!} \\
    &= \frac{2^{k-r}\left((r-1)\binom{n-\frac{r}{2}}{k+1-r} + (n-r+1)\binom{n-\frac{r}{2}}{k+1-r}\right)}{(r-1)!!(2k-r)!!}\\
    &= \frac{2^{k-r}\cdot n}{(r-1)!!(2k-r)!!}\binom{n-\frac{r}{2}}{k+1-r}.
\end{align*}
Dividing by $n$ shows that $E'(n,k,r)$ matches the claimed formula, which completes the proof.
\end{proof}

\end{document}